\documentclass[11pt, reqno]{amsart}

\usepackage[noend]{algpseudocode}
\usepackage{algorithmicx,algorithm}
\usepackage{fullpage}

\usepackage[colorlinks,
            linkcolor=red,
            anchorcolor=blue,
            citecolor=green
            ]{hyperref}
\usepackage{cite}

\usepackage{bm}
\usepackage{amsmath}
\usepackage{amssymb}
\usepackage{amsfonts}
\usepackage{amsthm}
\usepackage{mathrsfs}
  \theoremstyle{plain} \newtheorem{thm}{Theorem}[section]
  \theoremstyle{plain} \newtheorem{lem}[thm]{Lemma}
  \theoremstyle{plain} \newtheorem{prop}[thm]{Proposition}
  \theoremstyle{plain} \newtheorem{coro}[thm]{Corollary}
  \theoremstyle{plain} 
  \theoremstyle{plain} 
  
  \theoremstyle{remark} \newtheorem{rmk}[thm]{Remark}
  \theoremstyle{remark} 
  \theoremstyle{remark}

\usepackage{paralist}

\newcommand{\beq}{\begin{equation}}
\newcommand{\eeq}{\end{equation}}
\newcommand{\bals}{\begin{align*}}  
\newcommand{\bal}{\begin{align}}
\newcommand{\cu}{\textbf}
\newcommand{\real}{\mathbb{R}}
\newcommand{\cpx}{\mathbb{C}}
\newcommand{\intg}{\mathbb{Z}}
\newcommand{\posint}{\mathbb{N}}

\usepackage{epsfig}

\usepackage{multirow}
\usepackage{subfigure}
\usepackage{graphicx}
\usepackage{float}
\usepackage{booktabs}
\usepackage{enumerate}

\numberwithin{equation}{section}

\usepackage{fancyhdr}

\title[On Threshold Solutions of equivariant CSS]{On Threshold Solutions of equivariant Chern-Simons-Schr\"odinger Equation}

\author{Zexing Li}
\address{School of Mathematic Science\\
Peking University\\
Beijing\\ China}
\email{lizexing@pku.edu.cn}
\author{Baoping Liu}
\address{Beijing International Center for Mathematical Research\\
 Peking University\\ 
 Beijing\\
  China}
\email{baoping@math.pku.edu.cn}

\begin{document}
\maketitle

\begin{abstract}
We consider the self-dual Chern-Simons-Schr\"odinger model in two spatial dimensions. This problem is $L^2$-critical. Under equivariant setting, global wellposedness and scattering were proved in \cite{liu2016global} for solution with initial charge below certain threshold given by the ground state. In this work, we show that the only non-scattering solutions with threshold charge are exactly the ground state up to scaling, phase rotation and the pseudoconformal transformation. We also obtain partial result for non-self-dual system.
\end{abstract}

\tableofcontents

\section{Introduction}

\subsection{Covariant formulation}
The Chern-Simons-Schr\"odinger equation is a nonrelativistic quantum model describing the dynamics of a large number of charged particles in the plane interacting both directly and via a self-generated electromagnetic field. The model is a Lagrangian field theory on $\real^{1+2}$ associated to the action 
\[ L[A, \phi] = \frac{1}{2} \int_{\real^{1+2}} \left[ \text{Im}(\bar{\phi}\bm{D}_t \phi) + | \bm{D}_x \phi|^2 - \frac{g}{2} |\phi|^4 \right] dxdt + \frac{1}{2} \int_{\real^{1+2}} A \wedge dA \]
Here, $\phi: \real^{1+2} \rightarrow \cpx$ is a scalar field describing the particle system, the potential $A:=A_0 dt + A_1 dx_1 + A_2 dx_2$ is a real-valued 1-form on $\real^{1+2}$, the associated covariant differentiation operators $\bm{D}_\alpha:= \partial_\alpha + i A_\alpha$ for $\alpha \in \{ 0, 1, 2 \}$ are defined in terms of the potential $A$, and $g \in \real$ is a coupling constant. For indices, we use $\alpha = 0$ for time variable and $\alpha = 1, 2$ for spacial variable $x_1, x_2$. The Lagrangian is invariant with respect to the transformation 
\beq \phi \mapsto e^{-i\theta}\phi \quad A \mapsto A + d\theta \label{1gaugeinv} \eeq
for compactly supported real-valued function $\theta(t, x)$.

Computing the Euler-Lagrange equation, we obtain the Chern-Simons-Schr\"odinger equation (CSS)
\beq \left\{ \begin{array}{rl}
\bm{D}_t \phi &= i \bm{D}_l \bm{D}_l \phi + i g |\phi|^2 \phi \\
F_{01} &= -\text{Im} (\bar{\phi} \bm{D}_2 \phi) \\
F_{01} &= \text{Im} (\bar{\phi} \bm{D}_1 \phi) \\
F_{12} &= -\frac{1}{2} |\phi|^2
\end{array} \right. \label{1CSSoriginal}\tag{CSS}
\eeq
where $F = dA$ is the curvature 2-form, namely $F_{\alpha \beta} = \partial_\alpha A_\beta - \partial_\beta A_\alpha$. The system (\ref{1CSSoriginal}) is a basic model of Chern-Simons dynamics \cite{ezawa1991breathing,ezawa1991nonrelativistic,jackiw1991time}. For further physical motivation to study (\ref{1CSSoriginal}), such as quantum Hall effect, high temperature superconductivity and the quantization of Heisenberg ferromagnets, see \cite{deser1982topologically,jackiw1981super,jackiw1991topological,martina1993self,wilczek1990fractional}.

We have conservation laws for charge and energy
 \begin{align}
  \text{chg}[\phi]:=& \int_{\real^2} |\phi|^2 dx,\\
  E[\phi] :=& \int_{\real^2} \left( \frac{1}{2} |\bm{D}_x \phi|^2 - \frac{g}{4} |\phi|^4 \right) dx.
\end{align}
The system is $L^2$-critical in the sense that it admits a scaling transformation leaving the charge of $\phi$ and the equation invariant.
  \beq
  (\phi, A) \mapsto  \left\{ \begin{array}{rl}
\tilde{\phi} (t, x)&:= \lambda \phi (\lambda^2 t, \lambda x),\\
\tilde{A_0} (t, x)&:= \lambda^2 A_0 (\lambda^2 t, \lambda x),\\
\tilde{A_j} (t, x)&:= \lambda A_j(\lambda^2 t, \lambda x).
\end{array} \right. \label{l2scaling}
  \eeq

The property of this system changes when $g$ varies. Via Bogomol'nyi identity (\ref{bogomol}), the dividing point is the \emph{self-dual} case $g=1$, where the energy functional enjoys a complete square structure (\ref{squareenergy}). Generally speaking, self-duality refers to theories in which interactions have particular forms and special strengths such that the second order equation of motion reduces to the first which are simpler to analyze. This feature draws crucial physical importance to  models like self-dual Yang-Mills theory, self-dual Yang-Mills-Higgs theory and   self-dual Chern-Simons theory \cite{dunne2009self}.

In this paper, we impose Coulomb gauge and restrict to the equivariant setting. We first rewrite (\ref{1CSSoriginal}) in the polar coordinates of $\real^2$. Define
 \begin{align*}
   \partial_r = \frac{x_1}{|x|} \partial_1 + \frac{x_2}{|x|}\partial_2,&\qquad \partial_\theta= -x_2 \partial_1 + x_1 \partial_2.
 \end{align*}
 Correspondingly we define 
 \begin{align*}
   A_r = \frac{x_1}{|x|} A_1 + \frac{x_2}{|x|}A_2,&\qquad A_\theta= -x_2 A_1 + x_1 A_2,\\
   \bm{D}_r = \frac{x_1}{|x|}  \bm{D}_1 + \frac{x_2}{|x|} \bm{D}_2,&\qquad \bm{D}_\theta= -x_2  \bm{D}_1 + x_1  \bm{D}_2.   \end{align*}
We can formulate (\ref{1CSSoriginal}) equivalently as
\beq \left\{ \begin{array}{rl}
\bm{D}_t \phi &= i \left( \bm{D}_r^2 +  \frac{1}{r} \bm{D}_r + \frac{1}{r^2} \bm{D}_\theta^2 \right)\phi+ i g |\phi|^2 \phi \\
\partial_t A_r - \partial_r A_0  &= -\frac{1}{r} \text{Im} (\bar{\phi} \bm{D}_\theta \phi) \\
\partial_t A_\theta - \partial_\theta A_0 &= r\text{Im} (\bar{\phi} \bm{D}_r \phi) \\
\partial_r A_\theta - \partial_\theta A_r &= -\frac{1}{2} r |\phi|^2
\end{array} \right.\label{1csspol}
\eeq
with energy taking the form
\beq  
  E[\phi] = \int_{\real^2} \left( \frac{1}{2} |\bm{D}_r \phi|^2 + \frac{1}{2r^2}|\bm{D}_\theta \phi|^2 - \frac{g}{4} |\phi|^4 \right) dx.\label{energy-r}
\eeq

Now we introduce the $m$-equivariant $(m \in \intg)$ ansatz\footnote{We will often denote $u$ the radial part of $\phi$, and won't distinguish unless necessary. We also remark that the equivariant assumption involves the radial case as $m=0$.}.
\beq 
\begin{split}
\phi(t, x) = e^{im\theta}u(t, r), &\quad A_1(t, x) = -\frac{x_2}{r}v(t, r), \\ 
A_2(t, x) = \frac{x_1}{r}v(t, r),& \quad A_0 (t, x) = w(t, r).  \label{mequi}
\end{split}
\eeq
The equivariant solutions of Chern-Simons-Schr\"odinger system are called vortex solutions and appear in various related physical contexts (for instance \cite{paul1986charged, jackiw1990self, de1986electrically}). In addition, as a reasonable and effective simplification, equivariant reduction is also applied commonly to other geometric equations, for example Chern-Simons-Higgs \cite{chen2009symmetric}, wave map \cite{struwe2003equivariant,cote2015characterization} and Schr\"odinger map \cite{chang2000schrodinger,bejenaru2013equivariant}. Also note that our formulation (\ref{mequi}) implicitly indicate that we have chosen the Coulomb gauge condition\footnote{Conversely, this ansatz can be derived from the Coulomb gauge condition plus the equivariant assumption merely on $\phi$. See \cite{kim2019css} for details.} 
 \beq \nabla \cdot A_x = 0. \label{1coulomb} \eeq
Then (\ref{mequi}) and (\ref{1csspol}) imply that 
\beq A_r =0, \quad \partial_r A_0 = \frac{1}{r}\left( m + A_\theta \right)|\phi|^2,\quad \partial_r A_\theta = - \frac{1}{2}r |\phi|^2.\label{1ar0} \eeq
We make the natural boundary condition that $A_0$ decays to zero at spatial infinity (see \cite{berge1995blowing} for further discussion). Hence, we obtain explicit formula for $A_\theta$ and $A_0$
\begin{align}
  A_\theta [u](t, r) &= -\frac{1}{2} \int_{0}^r |u(t, s)|^2 sds,\\
  A_0 [u](t, r) &= -\int^\infty_r (m + A_\theta[u](t, s))|u(t, s)|^2 \frac{ds}{s}.
\end{align}
Now we can rewrite the Chern-Simons-Schr\"odinger equation under the $m$-equivariant assumption as the following $\phi$-evolution
\beq (i\partial_t + \Delta)\phi = \frac{2m}{r^2} A_\theta \phi + A_0 \phi + \frac{1}{r^2}A_\theta^2 \phi - g |\phi|^2\phi, \label{CSS} \tag{eCSS}
\eeq
or the $u$-evolution
\beq (i\partial_t + \Delta_m)u = \frac{2m}{r^2} A_\theta u + A_0 u + \frac{1}{r^2}A_\theta^2 u - g |u|^2 u. \label{uCSS}\eeq
where 
\beq \Delta_m := \partial_{r}^2 + \frac{1}{r} \partial_r - \frac{m^2}{r^2} \eeq
 is the Laplacian for $m$-equivariant functions in $\real^2$. Also, we denote the nonlinear part by 
 \beq F(\phi) := \frac{2m}{r^2} A_\theta \phi + A_0 \phi + \frac{1}{r^2}A_\theta^2 \phi - g |\phi|^2 \phi\label{nl} \eeq
 which is still an $m$-equivariant function. 
In this article, we will focus on (\ref{CSS}). We will further restrict to the physically relavant cases $m \ge 0$ \cite{dunne2009self}.  

\subsection{Known results and threshold problem}

Chern-Simons-Schr\"odinger system (\ref{1CSSoriginal}) has drawn much attention since the 90's. Under the Coulomb gauge, local wellposedness was first established with initial data in $H^2$ by Berg\'e-de Bouard-Saut \cite{berge1995blowing}. For $H^1$ initial data with small charge, they also obtained global existence(but without uniqueness). Huh \cite{huh2013energy} showed that (\ref{1CSSoriginal}) has a unique local-in-time solution for $H^1$ data, without continuous dependence.    Lim \cite{lim2018large} obtained $H^1$ local well-posedneess with weak Lipschitz dependence for small $L^2$ data. Using heat gauge, Liu-Smith-Tataru \cite{liu2014local} established local well-posedness and strong Lipschitz dependence in $H^\epsilon$, $\epsilon > 0$ for small $H^\epsilon$ data. In addition, Oh-Pusateri \cite{oh2015decay} prove global existence and scattering for solutions with small data in weighted Sobolev spaces,  by revealing a cubic null structure under the Coulomb gauge.
So far, wellposedness  for (\ref{1CSSoriginal})  at the critical regularity in any gauge remains an interesting open problem. 

Under equivalence setting, Liu-Smith \cite{liu2016global} demonstrated that the local wellposedness of (\ref{CSS}) with $L^2$ data can be proved via mere Strichartz estimates. Moreover,  a threshold result is obtained in  \cite{liu2016global}.

To explain the result, we first note that for $g \ge 1$, (\ref{CSS}) admits soliton solutions. Consider the elliptic equation
\beq \Delta_m u - \alpha u - \frac{2m}{r^2} A_\theta[u] u - A_0[u]u- \frac{1}{r^2}A_\theta[u]^2 u + g |u|^2 u = 0 \label{station}\eeq
with $\alpha \ge 0$.
When $g = 1$, Byeon-Huh-Seok \cite{byeon2012gaugeNLS,byeon2016standing} showed that (\ref{station}) admits a unique positive\footnote{In fact, when $m > 0$, we only have $Q > 0$ on $\real^2 \backslash \{ 0\}$. Indeed the zero at the origin exists for any $m$-equivariant function $f$ to be continuous. So in the following text, when we say a $m$-equivariant ($m > 0$) function is positive, we will always mean it's positive in $\real^2 \backslash \{ 0\}$.} radial solution with $\alpha = 0$
\beq Q^{(m)}(r) := \sqrt{8} (m+1) \frac{r^m}{1 + r^{2(m+1)}} \label{solitonq} \eeq
 under the boundary condition $A_0 \rightarrow 0$ as $|x| \rightarrow +\infty$. We simplify the notation by writing $Q=Q^{(m)}$ when $m$ is fixed.  This generates the static solution 
\beq \phi^{(m)}(x) := e^{im\theta}Q(r)\label{solitonpsi}\eeq
to self-dual (\ref{CSS})\footnote{In the self-dual case, with the help of (\ref{squareenergy}),  one can show that  zero energy solutions are gauge equivalent to static solutions of (\ref{1CSSoriginal}), even without equivariant assumption. Rigorous proof can be found in~\cite{HuSe13,kim2019css}.}.

In \cite{byeon2012gaugeNLS,byeon2016standing}, the authors also proved the non-existence of solution  for (\ref{station})  when $g \in (0, 1)$ and the existence of positive radial solution for (\ref{station}) with $\alpha \ge 0$ as $g > 1$, which we denote as  $Q^{(m, g,\alpha)}$ \footnote{Since $\alpha$ depends on $m$ and $g$, we also use the notation $Q^{(m, g)} := Q^{(m, g, \alpha(m, g))}$ and similar for $\phi$, $\psi$.}. By writing  $\phi^{(m, g,\alpha)}(x) := Q^{(m, g,\alpha)}(r) e^{im\theta}$, we obtain either a static solution ($\alpha = 0$) or a stationary wave ($\alpha > 0$) to (\ref{CSS}) for $g>1$, which is of the form  $\psi^{(m, g,\alpha)}(t, x):= \phi^{(m, g,\alpha)}(x) e^{i\alpha t}$ for some $\alpha \ge 0$. It is also conjectured in \cite{byeon2012gaugeNLS} that (\ref{CSS}) only admits stationary wave when $g > 1$.


In fact, these soliton solutions are the minimal-charge obstructions to global well-posedness and scattering as explained in the following threshold theorem. 
Let us first define the equivariant Sobolev space as
\[ H^s_m := \{ f \in H^s : \exists u = u(r)\, \text{s.t.} \,f(x) = e^{im\theta}u(r) \}, \quad L^2_m := H^0_m. \]
The homogeneous Sobolev space $\dot{H}^s_m$ is also defined in this way. It is easy to see that $\| f \|_{\dot{H}^1_{m}}^2 = \| \partial_r f \|_{L^2}^2 + \| \frac{m}{r} f \|_{L^2}^2$. 

\begin{thm}[Threshold result \cite{liu2016global}]\label{thresholdresult}
    Let $m \in \posint  := \{ n \in \intg: n \ge 0 \}$.
    \begin{enumerate}[(1)]
      \item Let $g < 1$. Then for any initial data $\phi_0 \in L^2_m$, the solution $\phi$ of (\ref{CSS}) is global-wellposed and scatters both forward and backward in time.
      \item Let $g = 1$. Then for any initial data $\phi_0 \in L^2_m$ with $\mathrm{chg}(\phi_0) < \mathrm{chg}(Q^{(m)}) = 8\pi (m+1) $, the solution $\phi$ of (\ref{CSS}) is global-wellposed and scatters both forward and backward in time.
      \item Let $g > 1$. Then there exists a constant $c_{m, g} > 0$ such that for any initial data $\phi_0 \in L^2_m$ with $\mathrm{chg}(\phi_0) < c_{m, g} $, the solution $\phi$ of (\ref{CSS}) is global-wellposed and scatters both forward and backward in time. Moreover, the minimum charge of a nontrivial standing wave solution $\psi^{(m, g)}$ in the class $L^\infty_t L^2_m$ is equal to $c_{m, g}$. 
    \end{enumerate}
\end{thm}

\begin{rmk} By scattering forward/backward in time,  we mean there exists  $\phi_{\pm}\in L^2$, such that \[\lim_{t \to \pm\infty} \| \phi(t) - e^{i t \Delta} \phi_{\pm} \|_{L^2} = 0.\]
In the proof of Theorem~\ref{thresholdresult}, the $L^4_{t,x}$ norm plays the role of a scattering norm in the following sense.  Let $\phi:I\times \mathbb{R}^2\rightarrow \mathbb{C}$ be a solution to (\ref{CSS}), where $I$ is the maximal lifespan, if $\sup I=+\infty$ and $\|\phi\|_{\|\phi\|_{L^4_{t,x}([0,+\infty)\times \mathbb{R}^2)}}<\infty$, then solution scatters forward in time. 

For this reason, we say a solution $\phi$, with maximal lifespan $I$,  blows up forward or backward in time if $\|\phi\|_{L^4_{t,x}(I_{\pm}\times \mathbb{R}^2)}=\infty$, with $I_+=[0, \sup I)$ and $I_{-}=(\inf I, 0]$.  In particular, it contains two scenarios: to blow up at infinite time or at finite time.
\end{rmk}

Note that (\ref{CSS}) admits pseudoconformal symmetry. The pseudoconformal transformation\footnote{The pseudoconformal invariance actually holds for the general   (\ref{1CSSoriginal}) \cite{huh2009blow}.} 
\beq PC_T: \quad \psi(t, x) \mapsto \frac{1}{T-t} e^{-i\frac{|x|^2}{4(T-t)}} \psi\left( \frac{t}{T(T-t)}, \frac{x}{T-t} \right), \qquad \forall \, t < T. \label{pct1}\eeq
keeps the equation (\ref{CSS}) invariant and conserves the solution's charge. By applying it to the solitons $\psi^{(m,g)}$ ($g \ge 1$) as above, we get $PC_T[\psi^{(m, g)}]$, another $m$-equivariant solution for (\ref{CSS}) with the threshold charge. It blows up at a finite time $t = T$, while in contrast, the soliton $\psi^{(m,g)}$ blows up at infinite time.

It is a natural question to study solutions  above or at the threshold charge. Recently,  Kim-Kwon~\cite{kim2019css,kim2020construction} studied finite time blow up solutions for the self-dual (\ref{1CSSoriginal}) under equivariant setting ($m\geq 1$). They constructed  pseudoconformal  blow-up solution with given asymptotic profile and studied its instability mechanism.  Furthermore, they constructed a co-dimension 1 manifold yielding pseudoconformal blow-up solutions.  Kim-Kwon-Oh~\cite{kim2020blow} considered the radial case and constructed data set that leads to blow up solutions whose blow-up rate differs from the pseudoconformal rate by a power of logarithm.  

On the other hand, our work focus on the special role $PC_T[\psi^{(m, g)}]$ plays. We present a characterization for $H^1_m$ solutions with exact threshold charge.

\subsection{Main result}

Our main result is that, in self-dual case $g=1$, any blowup $H^1_m$ solution must be  (\ref{solitonpsi}) up to symmetries. 

  \begin{thm}[Characterization of threshold solution in self-dual case]\label{char} For $m \ge 1$, $g = 1$ and initial data $\phi_0 \in H^1_m (\real^2)$, $\|\phi_0\|_{L^2} = \|Q^{(m)}\|_{L^2}$, one of the following three scenarios happens:
  \begin{enumerate}[(1)]
    \item $u$ equals to pseudoconformal transformation of the ground state $Q^{(m)}$ up to phase rotation and scaling. 
    \item $u$ equals to the ground state $Q^{(m)}$ up to phase rotation and scaling.
    \item $u$ scatters both forward and backward in time.
  \end{enumerate}
  And for $m = 0$, $g = 1$ and initial data $\phi_0 \in H^1_{rad} (\real^2)$, $\|\phi_0\|_{L^2} = \|Q^{0}\|_{L^2}$, only cases (2), (3) with $m =0$ can happen. In particular, the solution must exist globally.
  \end{thm}

  Noticing that a non-scattering solution blows up either at infinite time or at finite time, this result comes down to the following two theorems. 
 
  \begin{thm}[Rigidity of blowup in finite time in self-dual case]\label{rigidfinselfdual}
    For $m \in \posint$,  $g = 1$, and initial data $\phi_0 \in H^1_m (\real^2)$, $\|\phi_0\|_{L^2} = \|Q^{(m)}\|_{L^2}$, if the solution of (\ref{CSS}) $\phi$ blows up at $T > 0$, i.e. $\|\phi\|_{L^4_{t,x}([0, T) \times \real^2)} = \infty$,
    then $\exists \gamma \in [0, 2\pi), \lambda \in \real_+$ s.t.
    \[ \phi(t, x) = e^{i\gamma} PC_T [\lambda \phi^{(m)}(\lambda \cdot)](t, x), \qquad \forall \, t < T. \]
  \end{thm}
  
  \begin{rmk}\label{H1blowup} It's easy to see from (\ref{solitonq}) and (\ref{pct1}) that $Q^{(m)} \in H^1_m $for all $m \ge 0$ and $PC_T[Q^{(m)}] \in H^1_m$ only for $m \ge 1$. So for $m=0$, Theorem \ref{rigidfinselfdual} indicates that any threshold solution generated by $\phi_0 \in H^1_{rad}(\real^2), \,\| \phi_0\|_{L^2} = \| Q^{(0)} \|_{L^2} $ cannot blow up in finite time, as is stated in Theorem \ref{char}.
\end{rmk}
  
  \begin{thm}[Rigidity of blowup in infinite time in self-dual case] \label{rigidinfselfdual}
  For $m \in \posint$, $g =1$ and initial data $\phi_0 \in H^1_m (\real^2)$, $\|\phi_0\|_{L^2} = \|Q^{(m)}\|_{L^2}$, if the solution of (\ref{CSS}) $\phi$ blows up in infinite time, say at $+\infty$, i.e. $\|\phi\|_{L^4_{t,x}([0, +\infty) \times \real^2)} = \infty$,
   then $\exists \gamma \in [0, 2\pi), \lambda \in \real_+$ s.t.
    \[ \phi(t, r) = e^{i\gamma} \lambda \phi^{(m)}(\lambda r). \] 
  \end{thm}

\begin{rmk} For $m \ge 1$, after restricting initial data to a smaller space $\Sigma := \{f \in H^1_m: |x|f \in L^2\}$, these two results are equivalent through pseudoconformal transform (\ref{pct1}), since pseudoconformal transform maps $\Sigma$ into itself. \end{rmk}

Since the threshold behavior appears as well for the non-self-dual case $g > 1$, we may expect similar rigidity for threshold solution. Here we present the result for finite-time blowup.

\begin{thm}[Rigidity of blowup in finite time for $g > 1$]\label{rigidfinnonselfdual}
    For $m \in \posint$ and $g>1$, if $\phi_0 \in H^1_m (\real^2)$, $\|u_0\|_{L^2} = c_{m, g}$, and the solution of (\ref{CSS}) $\phi$ blows up at $T > 0$, then $\exists \gamma \in [0, 2\pi), \lambda \in \real_+$ and a m-equivariant standing wave solution $\psi^{(m, g)}(t, x) = e^{i\alpha t}\phi^{(m,g)}(x)$ $(\alpha \ge 0)$ solving (\ref{CSS}), s.t.
    \[ \phi(t, x) = e^{i\gamma} PC_T [\lambda \psi^{(m, g)}(\lambda^2 \cdot, \lambda \cdot)](t, x), \qquad \forall \, t < T. \]
\end{thm}

\begin{rmk} Compared with Theorem \ref{rigidfinselfdual}, we don't know whether all the standing wave solutions are the same (up to symmetry). Also due to the lack of knowledge on uniqueness of soliton and spectral analysis of its perturbation, our current approach for Theorem \ref{rigidinfselfdual} cannot apply for $g > 1$ case. \end{rmk}

\begin{rmk} In non-self-dual case $g > 1$, we only know $PC_T [\psi^{(m, g,\alpha)}] \in H^1_{m} $ when $\alpha > 0$ from the exponential decay in Proposition \ref{varcharnonselfdual}. For $\alpha = 0$ case, we may not have $H^1_m$ finite-time blowup with solution $PC_T [\psi^{(m, g, 0)}]$ as in Remark \ref{H1blowup}. But such static solution is actually conjectured not to exist \cite{byeon2012gaugeNLS}. \end{rmk}

 (\ref{CSS}) with $g \ge 1$ can be viewed as a gauged version of the mass-critical focusing Schr\"odinger equation
 \beq  (i\partial_t + \Delta)u = - |u|^{\frac{4}{d}}u. \tag{NLS} \label{NLS}
\eeq 
It shares many essential features with (\ref{NLS}) like symmetries, conservation laws and soliton behaviors.  So it's worthwhile to review the results of (\ref{NLS}).
 
(\ref{NLS}) is also $L^2$-critical with pseudoconformal symmetry. It has a unique standing wave soliton \cite{berestycki1979existence, kwong1989uniqueness}  $e^{it}R(x)$, with $R(x)$ radial, positive and Schwartz, solving an elliptic equation
\beq \Delta R - R + R^3 = 0.\eeq
 Weinstein \cite{weinstein1983nonlinear}  proved that any $H^1$ initial data with mass less than $\|R\|_{L^2}$ will generate a global solution. Killip-Tao-Visan \cite{killip2007cubic} then showed  global wellposedness and scattering for radial data with mass below the threshold, and the higher dimension case was solved by \cite{killip2009mass}.
Finally, Dodson \cite{dodson2015global} extends this threshold result to general $L^2$ non-radial data for all dimensions.
 
Now we come to the threshold characterization results for (\ref{NLS}). Combined with virial argument and a rigidity result of Weinstein \cite{weinstein1986structure}, Merle \cite{merle1992uniqueness,merle1993determination} proved the rigidity of blowup at finite time for $H^1$ data with threshold mass. The proof was simplified by Hmidi-Karaani \cite{hmidi2005blowup} via profile decomposition. On infinite-time blowup, the first result owes to Killip-Li-Visan-Zhang \cite{killip2009characterization}, who showed that for $d \ge 4$ $H^1$ radial data (and later for $d = 2, 3$ in \cite{li2012rigidity}), the rigidity theorem like Theorem \ref{rigidinfselfdual} holds. Li-Zhang then developed a local iteration scheme to obtain additional regularity in \cite{li2010regularity}, which implies the rigidity result in $L^2(\real^d)$ for radial data as $d \ge 4$\cite{li2010regularity} and later for splitting-spherical symmetry as $d = 6$ \cite{li2010focusing}. Our work is in similar spirit with \cite{merle1993determination, killip2009characterization,li2012rigidity}.

However, we should also point out some differences between (\ref{CSS}) and (\ref{NLS}). Firstly,   (\ref{NLS}) admits standing wave $e^{it}R(x)$ with $R$  decaying exponentially, while ground state $Q$ for self-dual (\ref{CSS}) is static solution with only polynomial decay.  Secondly, ground state $R$ for (\ref{NLS}) serves as an extremizer of the Gagliardo-Nirenberg  interpolation inequality which is important in many of the compactness argument, but we don't have such characterization  for $Q$. Besides,
the non-local nonlinearity for (\ref{CSS}) makes the analysis more challenging, especially when we need to analyze the linearized operator around the ground state. On the other hand, we are lucky enough that
 the Bogomol'nyi operator and the self-duality structure of (\ref{CSS}) are of great help to overcome the new difficulties and prove our results.

At last, we would like to mention one more series of relevant results: characterization of threshold solution for energy-critical equations. The pioneering work is attributed to Duyckaerts-Merle \cite{duyckaerts2008dynamics,duyckaerts2009dynamic}. They characterized the threshold radial solution for $d = 3, 4, 5$ of energy-critical wave equation and Schr\"odinger equation by detailed spectral analysis, modulation method and concentration-compactness method. 
For subsequent related works, we refer to \cite{li2009dynamics,li2011dynamics,duyckaerts2010threshold,ibrahim2014threshold,miao2015dynamics,jendrej2018two}.

 \subsection{Outline of the proof}
    Our result consists of two parts, finite-time blowup (Theorem \ref{rigidfinselfdual} and Theorem \ref{rigidfinnonselfdual}) and infinite-time blowup (Theorem \ref{rigidinfselfdual}).  The  starting point is  the variational characterization of the ground state (see the elliptic theory in \S \ref{sec2.2}).  Then the proofs bifurcate since they rely on very different strategies. 

    \emph{1. Rigidity of Finite-time Blowup.}
     The proof for finite-time case follows the framework of \cite{hmidi2005blowup}, which serves as a simplification of Merle's original work \cite{merle1992uniqueness, merle1993determination}. We start with a sequential rigidity result in Proposition \ref{rigidseqselfdual}. Specifically, if a sequence of threshold charge functions blows up in $H^1$ norm with energy bounded, then they converge to   the soliton up to symmetry in $H^1$.   
For a solution $\phi$ blowing up at finite time $T$, this sequential rigidity implies that along a sequence of time there is charge concentration. We can then apply  truncated virial identity and explore the relation of virial quantity with energy (\ref{coorp}) to conclude $E(e^{i\frac{|x|^2}{4T}} \phi_0) = 0$, which forces $\phi$ to be the soliton after pseudoconformal transform.

We remark that   the Bogomol'nyi operator (\ref{bogomol-r}) and the $m$-equivariance condition help us to compensate for the lack of sharp Gagliardo-Nirenberg inequality. In fact our argument is even simpler than \cite{hmidi2005blowup}. Also, the above argument works for both self-dual and non-self-dual cases.

        \emph{2. Rigidity of Infinite-time Blowup.} In this case, the proof is more complicated.  The minimal blow up solution is characterized as having the compactness property, which is further illustrated as almost periodic modulo symmetry \cite{liu2016global}, see Theorem~\ref{almostpms}.  This property indicates a uniform localization of charge.  If we can further control the kinetic energy   uniformly small near infinity(Theorem \ref{stronglocal}),  we quickly reach a contradiction using   virial-type argument  if the energy is positive.         
        
        So the main difficulty reduces to the proof of  Theorem \ref{stronglocal}. We  proceed in the spirit of    \cite{killip2009characterization} and its improvement \cite{li2012rigidity}.  In  \cite{killip2009characterization}, Killip-Li-Visan-Zhang used the  in-out decomposition, weighted Strichartz estimates and non-scattering Duhamel's principle  to prove the uniform localization of kinetic energy for  minimal infinite time blowup solutions of (\ref{NLS}), for $d\geq 4$. The  high power of nonlinearity causes trouble in low dimension. To overcome the difficulty,   Li-Zhang \cite{li2012rigidity} used modulation analysis to prove  a weaker localization theorem (similar to Theorem~\ref{weaklocal}).  This technique requires a good understanding of the spectral information for linearized operator around ground state.   For (\ref{CSS}),  the linearized operator $\mathcal{L}_Q$ is non-local,  which is usually difficult to analyze. Luckily the self-duality provides good structure, and the spectral information is carefully studied in \cite{kim2019css}. Also there are more nonlocal terms to deal with, and that makes  this part of proof particuarly long and complicated.  
         
  The structure of this article is as follows. In Section \ref{sec2}, we recall the Bogomol'nyi operator, elliptic theory, truncated virial estimate as preparations. Section \ref{sec3} and \ref{sec4} deal with the finite-time case and the infinite-time case, respectively. Only the infinite-time blowup rigidity demands a great deal of harmonic analysis and spectral analysis for (\ref{CSS}), so we will record those tools therein. We remark that throughout the rest of this paper, we only consider non-self-dual case in \S \ref{3.3}.

\section{Preliminaries}\label{sec2}

\subsection{Notations}\label{sec2.1} Since we mainly work with $m$-equivaraint function $\phi(x)=e^{im\theta}u(r): \real^2 \rightarrow \cpx$, we usually refer to the radial part of such function $\phi$ as $u$. And we won't distinguish them as acted by functional or operators, if there's no confusion. For example, $A_\alpha [u]:= A_\alpha [\phi]$, $E[u]:= E[\phi, A[\phi]]$ and $\| u \|_{\dot{H}^1_m} := \| f \|_{\dot{H}^1_m} = \| f \|_{\dot{H}^1}$. 

We write $X\lesssim Y$ or $Y\gtrsim X$ to indicate $X\leq CY$ for
some constant $C>0$.  If $C$ depends upon some additional parameters, we will indicate this with subscripts. For example, $X\lesssim_\phi Y$ means $X\leq C(\phi) Y$.
We use $O(Y)$  to denote any quantity $X$ such that $|X|\lesssim Y$.




\subsection{Bogomol'nyi operator}

We first introduce the Bogomol'nyi operator
\beq \bm{D}_+ := D_1 + iD_2 = e^{i\theta} \left( \bm{D}_r + \frac{i}{r}\bm{D}_\theta \right) \eeq
Then we have Bogomol'nyi identity
\beq | \bm{D}_x \phi|^2 = | \bm{D}_+ \phi|^2 + \nabla \times \bm{J} - F_{12} |\phi|^2 \label{bogomol}\eeq
where $\bm{J} := (\mathrm{Im}(\bar{\phi} \bm{D}_1 \phi), \mathrm{Im}(\bar{\phi} \bm{D}_2 \phi))$. Using Green's formula, we can rewrite the energy functional
\beq E[\phi] = \int_{\real^2} \left( \frac{1}{2} |\bm{D}_x \phi|^2 - \frac{g}{4} |\phi|^4 \right) dx = \int_{\real^2} \left( \frac{1}{2} |\bm{D}_+ \phi|^2 + \frac{1-g}{4} |\phi|^4 \right)dx.\eeq
%
Under the equivariant ansatz (\ref{mequi}), the Bogomol'nyi operator takes the form 
\[ \bm{D}_+ \phi = \left[ \left(\partial_r - \frac{m+A_\theta}{r} \right) u \right] e^{i(m+1)\theta} \]
Taking the radial part, we also use $\bm{D}_+$ to denote its action
\beq \bm{D}_+ u = \left(\partial_r - \frac{m+A_\theta}{r} \right) u. \label{bogomol-r}\eeq

For self-dual case $g = 1$, the energy turns into
\beq E[\phi] = \int_{\real^2} \frac{1}{2} |\bm{D}_+ \phi|^2  dx.\label{squareenergy}\eeq
In this case, the minimizer of energy will satisfy a solvable first-order PDE $\bm{D}_+ u = 0$, implying the following variational characterization easily. 

\subsection{Variational characterization of the ground state}\label{sec2.2}

We record the variational characterization of ground states of (\ref{CSS}) from elliptic theory, for both self-dual case and non-self-dual case. For completeness, we give their proofs in Appendix \ref{AppA}.

\begin{prop}[Variational characterization in self-dual case, \cite{byeon2012gaugeNLS, byeon2016standing, kim2019css}]\label{varcharselfdual}
  Let $ g=1$. For $\phi_0 \in H^1_m(\real^2) - \{ 0\}$, we have $E[\phi_0] \ge 0$. Moreover, $E[\phi_0] = 0$ if and only if $\phi_0$ equal to $\phi^{(m)}$ in (\ref{solitonpsi}) up to $L^2$ scaling and phase rotation, which also implies that $\phi$ is a static solution of (\ref{CSS}). 
\end{prop}

\begin{prop}[Variational characterization in non-self-dual case, \cite{liu2016global}]\label{varcharnonselfdual}
  Let $g > 1$. For $\phi_0 \in H^1_m - \{ 0\}$, $\| \phi_0 \|_{L^2} \le c_{m, g}$, we have $E[\phi_0] \ge 0$. Moreover, if $E[\phi_0] = 0$, then $\|\phi_0 \|_{L^2} = c_{m, g}$ and there exists $\alpha \ge 0$ s.t. $\psi^{(m, g,\alpha)}(t, x) = e^{i\alpha t} \phi^{(m, g,\alpha)}(x)$ is a standing wave solution of (\ref{CSS}). We also know any solution $\phi^{(m, g,\alpha)} $ decays exponentially for $\alpha > 0$.
\end{prop}


We remark that Proposition \ref{varcharnonselfdual} is weaker than Proposition \ref{varcharselfdual} by lack of uniqueness of the ground state. This accounts for the difference between Theorem \ref{rigidfinselfdual} and Theorem \ref{rigidfinnonselfdual}.

\subsection{Truncated virial Identity}

 The general virial identity for (\ref{1CSSoriginal}) is computed in \cite{liu2016global}. We establish the truncated version through direct computation here.

\begin{prop}[Truncated virial identity]\label{trunvirialident} Let $(\phi, A_0, A_1, A_2)$ be a solution to (\ref{1CSSoriginal}), and $\chi \in C^\infty_{0,\mathrm{rad}} (\real^2)$, we have

\begin{align}
  \partial_t  &\int_{\real^2} \chi(r)|\phi|^2 dx  = 2 \iint \partial_r \chi \mathrm{Im}(\bar{\phi}\bm{D}_r \phi)   rdrd\theta \label{trunvirial1}\\
  \partial_t^2 &\int_{\real^2} \chi(r)|\phi|^2 dx = 2 \iint 2r\partial_r \chi |\bm{D}_r\phi|^2 +\left[ \frac{1}{r^2} \partial_r(r \partial_r \chi) - \partial_r (\frac{1}{r} \partial_r \chi)\right] |\bm{D}_\theta \phi|^2  \nonumber \\
  & - \frac{1}{2}g\partial_r (r\partial_r \chi) |\phi|^4  +  \left[ -\partial_r^3(r\partial_r \chi) 
  + \frac{1}{2}\partial_r^2 \partial_r \chi + \frac{1}{2}\partial(\frac{1}{r} \partial_r\chi) \right] |\phi|^2 drd\theta \label{trunvirial2}
\end{align}
 \end{prop}
 \begin{rmk}
In particular, when $\chi = |x|^2$ and $|x|\phi \in L^2$, a limiting argument implies the virial identity
\begin{align}
    \partial_t \int_{\real^2} |x|^2 |\phi|^2 &= 4 \int_{\real^2} \mathrm{Im} (\bar{\phi}r\bm{D}_r \phi),\\
    \partial_t ^2\int_{\real^2} |x|^2 |\phi|^2 &= 16 E[\phi] \label{virialstand}.
  \end{align} 
The quadratic structure implies its cooperation with energy
  \beq  8t^2E[e^{i \frac{|x|^2}{4t}} \phi(0)] = \int_{\real^2} |x|^2 |\phi(t, x)|^2, \label{coorp}
  \eeq
  via direct computation.
 
\end{rmk}

\begin{proof}
  We define the stress-energy tensor 
  \[ T_{00} = F_{\theta r}= \frac{1}{2}r |\phi|^2, \quad T_{0r} = F_{0\theta} = r \mathrm{Im} (\bar{\phi}\bm{D}_r \phi), \quad T_{0\theta} = F_{r0} = \frac{1}{r} \mathrm{Im}(\bar{\phi}\bm{D}_\theta \phi).
  \]
  From $dF = d^2 A = 0$, we have 
  \beq \partial_\alpha T_{0\alpha} = 0. \label{stressenergy} \eeq
  Recall \cite[Lemma 5.1]{liu2016global}
  \begin{align}
  \partial_t T_{0r} =& - (2 + 2r\partial_r) |\bm{D}_r \phi|^2 + \frac{1}{2} rg\partial_r |\phi|^4 \nonumber\\
   & + \frac{1}{2} \partial_r |\bm{D}_\theta \phi|^2 - \frac{2}{r} \partial_\theta \mathrm{Re} (\overline{\bm{D}_\theta \phi} \bm{D}_r \phi)\label{t0r} \\ 
    & + r\partial_r \left[ \frac{1}{r^2} \left( \frac{1}{2} \partial_\theta^2 |\phi|^2 - |\bm{D}_\theta \phi|^2 \right) \right] + \left( \frac{1}{2}r \partial_r^3 + \frac{1}{2} \partial_r^2 - \frac{1}{2r} \partial_r \right) |\phi|^2\nonumber.
  \end{align}  
  Now we apply (\ref{stressenergy}) to show (\ref{trunvirial1})
  \begin{align*} & \partial_t  \int_{\real^2} \chi(r)|\phi|^2 dx =  2 \iint \chi(r)\partial_t T_{00} drd\theta \\
   =& - 2 \iint \chi(r) (\partial_r T_{0r} + \partial_\theta T_{0\theta}) drd\theta   = 2 \iint \partial_r \chi T_{0r} drd\theta.  \end{align*}
Next we can invoke (\ref{t0r}) and take derivative of time again to get (\ref{trunvirial2}).
  \end{proof}
  
 If we take $\chi$ to be a smooth truncation of $|x|^2$, the computation within $\{|x| \le R \}$ remains the same as the standard virial identity (\ref{virialstand}). The following estimates follows immediately.
   
 \begin{coro}\label{trunvirialest}
 Let  $(\phi, A_0, A_1, A_2)$ be a solution to (\ref{1CSSoriginal}), and $\chi_R = R^2\chi(R^{-1} \cdot)$ is the smooth cutoff of $|x|^2$, with $\chi \in C^\infty_{0,rad} (\{|x| \le 2\})$ and 
  \[ \chi(x) = \chi(|x|) =
  \left\{ \begin{array}{rl}
  |x|^2, &  |x| < 1,\\
  0,& |x| \ge 2.
  \end{array}
  \right. \]
Then
  \begin{align}
    \partial_t \int_{\real^2} &\chi_R |\phi|^2 = 2 \int_{\real^2} \partial_r \chi_R \text{Im} (\bar{\phi}\bm{D}_r \phi), \label{tvi1}\\
    \partial_t ^2\int_{\real^2} \chi_R |\phi|^2 = 16 E[\phi] + O&\left(\frac{1}{R^2} \int_{|x| \ge R} |\phi|^2\right) + O\left( \int_{|x|\ge R} |\nabla \phi|^2 +  \int_{|x|\ge R} |\phi|^4 \right)\label{tvi2}
  \end{align}
  \end{coro}
\begin{rmk}By constructing $\chi$ through convolution, it's easy to verify that it satisfies
\beq |\nabla \chi_R (x)|^2 \lesssim \chi_R (x) \label{decayprop}\eeq
\end{rmk}

\section{Finite time blow up}\label{sec3}

In this section, we prove Theorem \ref{rigidfinselfdual} and Theorem \ref{rigidfinnonselfdual}. We first consider the self-dual case, and then the non-self-dual case with small modification.

\subsection{Rigidity of normalized sequence}

We show the rigidity of normalized $H^1$ blowup sequence. 

Recall 
\begin{align*} \| \bm{D}_x f \|_{L^2}^2 =& \|D_r f\|_{L^2}^2 + \|D_\theta f\|_{L^2}^2 =  \| \partial_r f \|^2_{L^2} + \left\| \frac{m+A_\theta[f]}{r} f \right\|^2_{L^2}, \\
E[f] =& \frac{1}{2}\|\bm{D}_+ f \|_{L^2}^2=\frac{1}{2}\| \bm{D}_x f\|^2_{L^2} -  \frac{1}{4}\|f\|_{L^4}^4,\end{align*}
and the scaling property
\beq \|\bm{D}_x (\lambda f(\lambda \cdot))\|_{L^2}^2 = \lambda^2 \| \bm{D}_x f\|^2_{L^2},\qquad \forall \,\lambda \in \real_+. \label{energyscaling}
\eeq

To begin with, we discuss some properties of $\| \bm{D}_x f\|^2_{L^2}$, showing it is to some extent equivalent to $\dot{H}^1_m$ norm. The first two lemmas are concerned with size and weak convergence. 

  \begin{lem}\label{equilem1}
  For $m \ge 1$, $f \in H^1_m$,
  \begin{align} \left\| \frac{1}{r} f \right\|_{L^2} \lesssim_{m, \|f\|_{L^2}} \left\| \frac{m+A_{\theta}[f]}{r} f \right\|_{L^2},  \label{lem311}\end{align}
  \begin{align}\left\| \frac{A_{\theta}[f]}{r} f \right\|_{L^2} \lesssim \|f\|_{L^2}^2 \|\partial_r f\|_{L^2}. \label{lem312}
  \end{align}
  As a consequence, for $m \ge 0$,  $f \in H^1_m$,
  \beq \|f\|^2_{\dot{H}^1_m(\real^2)} \sim_{m, \|f\|_{L^2}} \| \bm{D}_x f\|^2_{L^2} \label{lem313} \eeq
  \end{lem}
  
\begin{lem}\label{equilem2}
  For $m \ge 0$, $f_n$ and $f$ uniformly bounded in $H^1_m(\real^2)$, and $f_n \rightharpoonup f$ in $H^1$, we have
  \begin{align}
     \partial_r f_n &\rightharpoonup \partial_r f\qquad \text{in}\,\, L^2, \label{lem321} \\
     \frac{1}{r}A_\theta[f_n] f_n &\rightharpoonup \frac{1}{r} A_\theta[f]f \qquad \text{in}\,\, L^2, \label{lem322}
  \end{align}
  and for $m \ge 1$, 
  \beq   \frac{1}{r}f_n \rightharpoonup \frac{1}{r} f\qquad \text{in}\,\, L^2. \label{lem323} \eeq
  In particular,
  \beq \bm{D}_+ f_n = \partial_r f_n - \frac{m+A_\theta[f_n]}{r} f_n \rightharpoonup \bm{D}_+ f,\qquad \text{in}\,\, L^2  \label{lem324}
  \eeq
\end{lem}

The following lemma indicates that we can recover strong convergence in $H^1$ through $L^2$ convergence for $v_n$ and $\bm{D}_+ v_n$.

\begin{lem}\label{equilem3}
      Suppose we have $\{v_n\}$ and $v$ uniformly bounded in $H^1_m$, and 
    \begin{align*}
      v_n \rightharpoonup v\quad \text{in}\,\, H^1,&\qquad v_n \rightarrow v\quad \text{in}\,\, L^2,  \\
      \bm{D}_+ v_n \rightarrow \bm{D}_+ v\quad \text{in}\,\,L^2,&\qquad v_n \rightarrow v \quad \text{in} \,\, L^4.
    \end{align*}
    Then we will have
    \[ v_n \rightarrow v\qquad \text{in}\,\, H^1. \]
\end{lem}

We leave their proof in Appendix \ref{AppB}. Now we work on the rigidity of $H^1$ blowup sequence. Take a time sequence of $H^1$ blowup solution. By rescaling, we normalize the charge and $\dot{H}^1$ norm with the energy going to zero. The proposition shows that such sequence must converge to the ground state. 

\begin{prop}\label{rigidseqselfdual}
For $m \ge 0$, let $v_n \in H^1_m(\real^2)$, satisfying
\begin{align*} 
\| v_n \|_{L^2} = \| \phi^{(m)}\|_{L^2}, & \qquad \|v_n \|_{\dot{H}^1_m} = \|\phi^{(m)} \|_{\dot{H}^1_m}, \\
E[v_n] = \frac{1}{2}\|\bm{D}_+ v_n\|_{L^2}^2 & \rightarrow 0 \qquad \text{as} \,\,n \rightarrow \infty,
\end{align*}  
where $\phi^{(m)} = Qe^{im\theta} \in H^1_m$ is the static soliton.
Then there exists $\gamma \in [0, 2\pi)$ and a subsequence of $\{v_n\}$ (still denoted by $\{ v_n\}$) such that
\beq v_n \rightarrow e^{i\gamma} \phi^{(m)} \quad \text{in}\, H^1.
\eeq
\end{prop}

The following corollary is another equivalent way of stating this rigidity, which comes easily after a contradiction argument. It will be useful in modulation analysis to control the $H^1_m$ norm by energy (see Lemma \ref{rigidquan}).

\begin{coro}\label{rigidqual}
  For $m \ge 0$, let $f \in H^1_m(\real^2)$, satisfying
\begin{align*} 
\| f \|_{L^2} = \| \phi^{(m)}\|_{L^2},  \qquad \|f \|_{\dot{H}^1_m} = \|\phi^{(m)}\|_{\dot{H}^1_m}. 
\end{align*}
Then the energy $E[f] \ge 0$ and for any $\epsilon >0$, there exists $\delta > 0$ such that if $E[f] < \delta$, then there exists $\gamma \in [0, 2\pi)$ such that 
\[ \| \phi^{(m)}- e^{i\gamma} f \|_{H^1_m} < \epsilon. \]
\end{coro}

\begin{proof}[Proof of Proposition \ref{rigidseqselfdual}]
   $\{v_n\}$ is uniformly bounded in $H^1_m$, so we can extract a weakly convergent subsequence (still denote by $\{v_n\}$) and $v \in H^1_m$
   \beq v_n \rightharpoonup v,\qquad \text{in}\,\, H^1.\label{conv1} \eeq
   
   Since $H^1_{\text{rad}}(\real^2)$ compactly embeds in $L^p_{\text{rad}}(\real^2)$ for any $p \in (2, \infty)$, it's easy to see $H^1_m(\real^2)$ also compactly embeds in $L^p_m(\real^2)$. Hence
   \beq v_n \rightarrow v, \qquad \text{in} \, L^4_m. \label{conv2}\eeq
   Note that the normalization condition and Lemma \ref{equilem1} implies
   \[ \| \bm{D}_x v_n\|^2_{L^2} \gtrsim_m \|v_n\|_{\dot{H}^1_m}^2 =  \|Q\|_{\dot{H}^1_m}^2. \]
   So
   \beq  \|v\|_{L^4} = \lim_{n\rightarrow \infty}  \|v_n\|_{L^4} = \lim_{n\rightarrow \infty} (-4 E[v_n] + 2\| \bm{D}_x v_n\|^2_{L^2}) \gtrsim_m \|Q\|^2_{\dot{H}^1_m} >0. \label{nonzero} \eeq
   
   Also from Lemma \ref{equilem2}, we have 
   \beq E[v] = \frac{1}{2}\|\bm{D}_+ v\|^2_{L^2} \le \liminf_{n\rightarrow \infty}  \frac{1}{2}\|\bm{D}_+ v_n\|^2_{L^2} = 0. \eeq
   Note that (\ref{nonzero}) ensures that $v$ is a non-zero function. Proposition \ref{varcharselfdual} implies that $v$ is ground state up to symmetry. Namely, there exists $ \gamma \in [0, 2\pi),\,\lambda \in \real_+$, such that 
   \[ v = e^{i\gamma} \lambda Q(\lambda r) e^{im\theta}. \]
    Hence 
   \beq \|v\|_{L^2} =  \|Q\|_{L^2} = \|v_n \|_{L^2}. \eeq
   Together with the norm convergence of $\bm{D}_+ v_n$, we now have two more strong convergence
   \begin{align}  
   \bm{D}_+ v_n \rightarrow \bm{D}_+ v,&\qquad \text{in}\,\,L^2, \label{conv3} \\
   v_n \rightarrow v, &\qquad \text{in}\,\,L^2. \label{conv4}
   \end{align}
   From (\ref{conv1}), (\ref{conv2}), (\ref{conv3}) and (\ref{conv4}), we conclude $v_n\rightarrow v$ in $\dot{H}^1_m$
    through Lemma \ref{equilem3}, which also indicates the scaling parameter $\lambda = 1$, thus completing the proof.
   
   \end{proof}

\subsection{Proof of Theorem \ref{rigidfinselfdual}}

   First we need a Cauchy-Schwartz type estimate. This estimate is first introduced by Banica for (\ref{NLS}) in \cite{banica2004remarks}, and used in the proof of \cite{hmidi2005blowup} to derive a crucial ODE control of truncated virial identity. It still holds for general $H^1$ functions even without equivariant assumption.
   \begin{lem}[Cauchy-Schwartz type estimate]\label{cauchyschwartz}
     For $f \in H^1$ and all $R > 0$, we have have
     \beq \left| \int \partial_r \chi_R \text{Im} (\bar{f} \bm{D}_r f) dx \right| \le \left( 2E[f] \int |f|^2 |\partial_r \chi_R|^2 dx \right)^{\frac{1}{2}}.\label{cs-selfdual} \eeq
     where $E[f] = \frac{1}{2} \| \bm{D}_+ f \|_{L^2}^2$ is the energy for  (\ref{1CSSoriginal}) in the self-dual case.
   \end{lem}
   
   \begin{proof} Notice that energy have another expression in polar coordinate (\ref{energy-r}) with $g=1$. So 
   with positivity of energy,  for any $\alpha \in \real$, \begin{align*}
  0 \le E[e^{i\alpha \chi_R(x)}f] &= \frac{1}{2} \int | \bm{D}_r (e^{i\alpha \chi_R}f) |^2 + \frac{1}{r^2} |\bm{D}_\theta (e^{i \alpha \chi_R} f)|^2 - \frac{1}{2} |e^{i \alpha \chi_R} f|^4 dx \\
  &=  \frac{1}{2} \int | \bm{D}_r f + i\alpha \partial_r \chi_R f |^2 + \frac{1}{r^2} |\bm{D}_\theta f |^2 - \frac{1}{2} |f|^4 dx \\
  &=  \frac{\alpha^2}{2} \int |f|^2 |\partial_r \chi_R|^2 |^2 dx - \alpha  \int \partial_r \chi_R \text{Im}(\bar{f} \bm{D}_r f) dx + E[f]
\end{align*}
The positivity of this quadratic form implies negative discriminant, which is exactly (\ref{cs-selfdual}).

\end{proof}

Now we are ready to prove Theorem \ref{rigidfinselfdual}.
\begin{proof}[Proof of Theorem \ref{rigidfinselfdual}] We divide the proof into two steps. 

\cu{Step 1. Normalized rigidity and charge concentration behavior.}

From Sobolev embedding and regularity $H^1_m$, we see that finite-time blowup condition 
\[\| \phi \|_{L^4_{t, x}([0, T) \times \real^2)} = \infty\]
implies the $L^\infty_t H^1_x$ blowup. Namely, there exists a sequence of time $t_n \nearrow T$ such that 
\beq \| \phi(t_n)\|_{H^1_m} \rightarrow \infty,  \qquad \mathrm{as} \,\, n \rightarrow \infty.\label{step11} \eeq
We set 
 \[ \rho_n = \frac{\| \phi^{(m)} \|_{\dot{H}^1_m}}{\|\phi(t_n, \cdot) \|_{\dot{H}^1_m}},\quad \text{and}\quad v_n = \rho_n \phi(t_n, \rho_n x).\] 
 Then from (\ref{step11}) and charge conservation, $\rho_n \rightarrow 0$ as $n \rightarrow \infty$. The sequence $\{ v_n \}$ then satisfies
 \[ \| v_n \|_{L^2} = \| Q\|_{L^2}, \quad \| v_n \|_{\dot{H}^1_m} = \| \phi^{(m)} \|_{\dot{H}^1_m}. \]
 Furthermore, by conservation of the energy, 
 \[ E[v_n] = \rho_n^2 E[\phi_0] \rightarrow 0,\qquad \text{as}\,\,n \rightarrow \infty.\]
Hence, $\{ v_n\}$ satisfies the assumption of Proposition \ref{rigidseqselfdual}, we have
\beq e^{-i\gamma} \rho_n \phi(t_n, \rho_n x) = e^{-i\gamma} v_n \rightarrow Q, \qquad \text{in} \,\, H^1\, \,\text{as}\,\, n \rightarrow \infty, \eeq
for some $\gamma \in [0, 2\pi)$.

This easily implies the charge concentration behavior
\beq |\phi(t_n, x)|^2 - \| Q \|_{L^2}^2 \delta_{x=0} \rightarrow 0,\qquad \text{in} \,\,\mathcal{D}'(\real^2), \label{chargecon}\eeq
where $\mathcal{D}'(\real^2)$ is the distribution on $\real^2$, and $\delta_{x=0}$ is the delta functional.

\cu{Step 2. Truncated virial estimate.}

Denote the truncated virial quantity for $\phi$ to be 
\[ V_R(t) = \int \chi_R(x) |\phi(t, x)|^2 dx, \]
where $\chi_R$ is the smooth truncation of $|x|^2$ as in Corollary \ref{trunvirialest}. Lemma~\ref{cauchyschwartz} together with the bound (\ref{decayprop}) we get, for $t \in [0, T),$
\begin{align} |\partial_t V_R(t) | &= 2\left| \int \partial_r \chi_R \text{Im} (\bar{\phi(t)} \bm{D}_r \phi(t)) dx \right| \le 2\left( 2E[\phi(t)] \int |\phi(t)|^2 |\partial_r \chi_R|^2 dx \right)^{\frac{1}{2}} \nonumber\\
&\le C E[\phi_0]^{\frac{1}{2}} \left( \int |\phi(t)|^2  \chi_R dx \right)^{\frac{1}{2}}  \lesssim_{\phi_0} (V_R(t))^{\frac{1}{2}}. \label{virialode} \end{align}
By integration we obtain, for every $t \in [0,T)$, 
\[ |(V_R(t))^{\frac{1}{2}} - (V_R(t_n))^{\frac{1}{2}}| \le C(\phi_0) |t_n - t|. \]
Now take $n \rightarrow \infty$, from (\ref{chargecon}), we get 
\[ |V_R(t)| \le C(\phi_0) (T-t)^2. \]
Noting that this bound is independent on $R$, by taking $R \rightarrow \infty$ we can see that virial quantity is also controlled by the same bound. Hence (\ref{coorp}) provides us with
\[  8t^2E[e^{i \frac{|x|^2}{4t}} \phi(0)] = \int_{\real^2} |x|^2 |\phi(t, x)|^2 \le C(\phi_0) (T-t)^2. \] 
Let $t \nearrow T$, we get
\[ E[e^{i \frac{|x|^2}{4T}}\phi_0] = 0. \]
Thus Proposition \ref{varcharselfdual} indicates that there exists $\gamma \in [0, 2\pi)$, $\lambda \in \real_+$, such that
\[ \phi_0 = e^{i\gamma} \lambda e^{-i\frac{|x|^2}{4T}} Q(\lambda r) e^{im\theta} = e^{i\gamma} PC_T [\tilde{\lambda} \phi^{(m)}(\tilde{\lambda} \cdot)](0, r), \] 
with $\tilde{\lambda} = \lambda T$. This conclude the proof of Theorem \ref{rigidfinselfdual}.
\end{proof}

\subsection{Non-self-dual case}\label{3.3}

In this subsection, we prove Theorem \ref{rigidfinnonselfdual} via same strategy as self-dual case.  

\begin{prop}\label{rigidseqnonselfdual}
For $m \ge 0$, $g > 1$, let $v_n \in H^1_m(\real^2)$, satisfying
\begin{align*} 
\| v_n \|_{L^2} = c_{m, g},  & \qquad \| v_n \|_{\dot{H}^1} = M \\
E[v_n] = \frac{1}{2}\|\bm{D}_+ v_n\|_{L^2}^2& - \frac{g - 1}{4} \| v_n \|_{L^4}^4  \rightarrow 0 \qquad \text{as} \,\,n \rightarrow \infty.
\end{align*}  
Then there exists $\gamma \in [0, 2\pi)$, a subsequence of $\{v_n\}$ (still denoted by $\{ v_n\}$) and $v \in H^1_m$ such that
\beq v_n \rightarrow v \quad \text{in}\, H^1,
\eeq
and $\psi(t, x) = e^{i\alpha t} v(x)$ for some  $\alpha \in \real \backslash\{ 0\}$ is a standing wave solution of (\ref{CSS}).
\end{prop}

\begin{proof}
    By uniform boundedness in $H^1$-norm and compact embedding, we know there exists $v \in H^1_m$ s.t. 
    \[ v_n \rightharpoonup v\quad \text{in} \,H^1,\qquad v_n \rightarrow v\quad \text{in} \, L^4. \]
    Again, by lower bound of $\| \bm{D}_x v_n\|_{L^2}$ from Lemma \ref{equilem1} and $L^4$ convergence, we know $v$ is nontrivial. And from Lemma \ref{equilem2}, 
    \beq E[v] = \frac{1}{2}\|\bm{D}_+ v\|_{L^2}^2 - \frac{g - 1}{4} \| v\|_{L^4}^4  \le \liminf_{n \rightarrow \infty} \frac{1}{2}\|\bm{D}_+ v_n\|_{L^2}^2 - \frac{g - 1}{4} \| v_n \|_{L^4}^4  = 0 \label{nsdd1} \eeq
    Note that 
    \beq \| v\|_{L^2} \le \liminf_{n\rightarrow \infty} \| v_n \|_{L^2} = c_{m, g}, \label{nsdd2} \eeq
    And Proposition \ref{varcharnonselfdual} forces $E[v] \ge 0$. So  $E[v] = 0$ and hence $v$ is a standing wave solution to (\ref{CSS}) with critical charge $c_{m, g}$. The norm convergence of (\ref{nsdd1}) and (\ref{nsdd2}) implies that
    \[ v_n \rightarrow v  \quad \text{in} \, L^2, \qquad \bm{D}_+ v_n \rightarrow \bm{D}_+ v \quad \text{in} \, L^2. \]
    Now Lemma \ref{equilem3} ensures the $H^1$ convergence and completes the proof. 
\end{proof}

Since the positivity of energy holds for $\| f \|_{L^2} \le c_{m, g}$, we can prove the counterpart of Cauchy-Schwartz type estimate for non-self-dual case as Lemma \ref{cauchyschwartz} in the same way. We state this lemma and omit its proof.
   \begin{lem}
     Fix $m \ge 0$ and $g > 1$. For $f \in H^1$, $\| f \|_{L^2} \le c_{m, g}$ and all $R > 0$, we have
     \beq \left| \int \partial_r \chi_R \text{Im} (\bar{f} \bm{D}_r f) dx \right| \le \left( 2E[f] \int |f|^2 |\partial_r \chi_R|^2 dx \right)^{\frac{1}{2}}.\label{cs} \eeq
     where $E[f] = \frac{1}{2} \| \bm{D}_+ u \|_{L^2}^2 - \frac{g-1}{4}\| f \|_{L^4}^4$ is the energy for (\ref{1CSSoriginal}) in the non-self-dual case.
   \end{lem}

Then we finish the proof of Theorem \ref{rigidfinnonselfdual} by the   same argument as Theorem \ref{rigidfinselfdual}, using charge concentration argument and truncated virial estimate.

\section{Infinite time blow up}\label{sec4}

In this section, we prove Theorem \ref{rigidinfselfdual}. First in \S \ref{4.1}, we introduce several tools from harmonic analysis, including the concentration compactness Proposition \ref{almostpms}. Then we reduce the whole proof to a frequency decay estimate Proposition \ref{freqdecayest} in \S \ref{4.2} and prove that estimate in \S \ref{4.3} and \S \ref{4.4}. 

\subsection{Preliminary on Harmonic Analysis}\label{4.1}

\subsubsection{Basic harmonic analysis}

We introduce the Littlewood-Paley multipliers in the usual way. In particular, let  $\varphi \in C^\infty_{0,rad} (\{|x| \le 2\})$ and 
  \[ \varphi(x) = \varphi(|x|) =
  \left\{ \begin{array}{rl}
  1, &  |x| < 1,\\
  0,& |x| \ge 2,
  \end{array}
  \right. \]
  and its scaling for $R > 0$
  \[ \varphi_{\le R} := \varphi (R^{-1} \cdot) ,\quad \varphi_{> R} := 1- \varphi_{\le R}.\] 
  Then for each $N \in 2^\intg$, define 
  \[ \mathcal{F}(P_{\le N} f) (\xi) := \varphi_{\le N}(|\xi|)\hat{f}(\xi),\quad P_{> N} = 1-P_{\le N}, \quad P_N = P_{\le N} - P_{\le \frac{N}{2}} \]
  and the fattened Littlewood-Paley operators
\[ \tilde{P}_N := P_{N/2} + P_N + P_{2N}. \]
The basic estimate in Littlewood-Paley theory is the following Bernstein estimate.
\begin{lem}[Bernstein estimates]
For $1 \le p \le q \le \infty$,
\begin{align*}
  \| |\nabla|^{\pm s} P_N f \|_{L^p(\real^d)} \sim & N^{\pm s} \| P_N f \|_{L^p(\real^d)} \\
  \| P_{\le N} f \|_{L^p(\real^d)} \lesssim & N^{\frac{d}{p} - \frac{d}{q}} \| P_{\le N} f \|_{L^p(\real^d)} \\
  \| P_N f \|_{L^p(\real^d)} \lesssim & N^{\frac{d}{p} - \frac{d}{q}} \| P_N f \|_{L^p(\real^d)}
\end{align*}
\end{lem}

While it's true that spatial cutoffs do not commute with Littlewood-Paley operators, the commutator isn't ``too bad", so that a weaker form of almost orthogonality still holds.

\begin{lem}[Mismatch estimates in physical space, \cite{killip2009characterization}]
Let $R, N > 0$. Then 
\begin{align*}
\| \varphi_{> R} P_{\le N} \varphi_{\le \frac{R}{2}} f \|_{L^p} \lesssim_m & N^{-m} R^{-m} \|f\|_{L^p}\\
  \| \varphi_{> R} \nabla P_{\le N} \varphi_{\le \frac{R}{2}} f \|_{L^p} \lesssim_m & N^{1-m} R^{-m} \|f\|_{L^p}   
\end{align*}
for any $1 \le p \le \infty$ and $m \ge 0$.
\end{lem}

Similar estimates hold when the roles of frequency and physical spaces are interchanged.

\begin{lem}[Mismatch estimate in frequency space, \cite{killip2009characterization}]\label{mismatch2}
For $R > 0$ and $N, M > 0$ such that $\max\{ N, M \} \ge 4 \min\{ N, M\},$
\begin{align*}
  \| P_N \varphi_{\le R} P_M f \|_{L^2} \lesssim_m & \max \{ N, M \}^{-m} R^{-m} \| f\|_{L^2} \\
  \| P_N \varphi_{\le R} \nabla P_M f \|_{L^2} \lesssim_m & M \max \{ N, M \}^{-m} R^{-m} \| f\|_{L^2}
\end{align*}
for any $ m \ge 0$. The same estimates hold if we replace $\varphi_{\le R}$ by $\varphi_{> R}$, or $P_N$ by $P_{< N}$ when $M \ge 4N $.
\end{lem}



\subsubsection{Strichartz estimates}

We present the classical linear Strichartz estimates. See for example \cite{cazenave2003semilinear}.

\begin{lem}[Strichartz estimates]
  Let $I$ be a time interval with $0 \in I$ and $u(0) = u_0 \in  L^2$ and $F \in L^{\frac{4}{3}}_{t}L^{\frac{4}{3}}_{x}(I \times \real^2)$.  Then the strong solution $u$ to the linear Schr\"odinger equation
  \[ u(t) := e^{it\Delta} u_0 - i\int^t_0 e^{i(t-t')\Delta} F(t') dt', \qquad t \in I,\]
  satisfies the estimate
  \[ \| u \|_{L^\infty_t L^2_x (I \times \real^2)} + \| u \|_{L^4_{t, x} (I \times \real^2)}  \lesssim \| u_0 \|_{L^2_x} + \| F \|_{L^{\frac{4}{3}}_{t, x} (I \times \real^2)} .\]
\end{lem}

The following weighted Strichartz estimate exploits heavily the equivariance in order to obtain spatial decay, which is originated from the radial case \cite{killip2007cubic}.

\begin{lem}[Weighted Strichartz estimates]\label{weightedstrichartz}
  For $m \ge 0$, let $I$ an interval, $t_0 \in I$, and let $F: I \times \real^2 \rightarrow \cpx$ be $m$-equivariant in space. Then
  \[ \left\| \int_{t_0}^t e^{i(t - t') \Delta} F(t') dt' \right\|_{L^2_x} \lesssim \left\| |x|^{-\frac{1}{2}} F \right\|_{L^{\frac{4}{3}}_t L^1_x}.\] 
\end{lem}

\begin{proof} As in the standard 
   proof of Strichartz estimate, besides $TT^*$ method and Hardy-Littlewood-Sobolev inequality, we only need to show a dispersive estimate
  \beq \| |x|^{\frac{1}{2}} e^{it\Delta} |x|^{\frac{1}{2}}f \|_{L^\infty_x(\real^2)} \lesssim |t|^{-\frac{1}{2}} \| f \|_{L^1_x(\real^2)} \label{mdisp} \eeq
  for any $m$-equivariant function $f(x) = u(r)e^{im\theta} \in L^1_m(\real^2)$. 
  And this follows from the computation of kernel of $e^{it\Delta}$ applied to $m$-equivariant functions. Under polar coordinates $x := (r \cos \theta, r \sin \theta)$, $y := (\rho \cos \alpha, \rho \sin \alpha)$, we see
  \begin{align*}
    (e^{it\Delta}f)(y)  & = \frac{1}{4\pi i t} \int_{\real^2} f(x) e^{\frac{i|x-y|^2}{4t}} dx \\
    &= \frac{1}{4\pi it} \int_0^\infty u(r) \int_0^{2\pi} e^{im\theta} e^{i\frac{r^2 + \rho^2 - 2r\rho \cos (\theta - \alpha)}{4t}} d\theta rdr \\
     & = \frac{1}{4\pi i t} e^{im\alpha} \int_0^\infty u(r) e^{i\frac{r^2+\rho^2}{4t}} \int_0^{2\pi} e^{i(m \omega - \frac{r \rho}{2t} \cos \omega)} d\omega rdr.
  \end{align*} 
  That is, for radial functions, $e^{-im\alpha}e^{it\Delta}e^{im\theta}$ has this kernel 
\begin{align}    [e^{-im\alpha}e^{it\Delta}e^{im\theta}](x, y) = &\frac{1}{4\pi it} e^{i\frac{r^2+\rho^2}{4t}} \int_0^{2\pi} e^{i(m \omega - \frac{r \rho}{2t} \cos \omega)} \frac{d\omega}{2\pi} \notag\\
    =& \frac{1}{4\pi i t} e^{i\frac{r^2+\rho^2}{4t}} J_m\left(\frac{r \rho}{2t}\right) \label{propagatordecay}
  \end{align}
  where $J_\nu$ denotes the Bessel function of order $\nu$. So by the behavior of Bessel functions~\cite{gradshteyn2014table}, $|(\ref{propagatordecay})|\lesssim | r \rho t|^{-\frac{1}{2}}$. (\ref{mdisp}) follows immediately.
\end{proof}

\subsubsection{In-out decomposition}

Finally, we present one more useful tool in the following analysis --- the incoming / outgoing decomposition developed in \cite{killip2007cubic, killip2009mass}, and m-equivariant version in \cite{liu2016global}. 
 It derives from the relationship between Bessel function and Fourier transform of radial function. For $f(r, \theta) = e^{im\theta} u(r) \in L^2_m$, 
 \[ \hat{f}(\rho, \alpha) = 2\pi (-i)^m e^{im\alpha} \int^\infty_0 J_m (\rho r) f(r) r dr \quad \]
where $J_\nu$ denotes the Bessel function of order $\nu$.  We split the Bessel function $J_m$ into two Hankel functions, $H^{(1)}_m$ and $H_m^{(2)}$, correspondingly to projections onto outgoing and incoming waves. In particular, 
\[
J_m(|x||\xi|)=\frac{1}{2} H_m^{(1)}(|x||\xi|) + \frac12 H_m^{(2)}(|x||\xi|)
\]
where $H_m^{(1)}$ is the order $m$ Hankel function of the first kind and $H_m^{(2)}$ 
is the order $m$ Hankel function of the second kind.
We can define the in-out decomposition\footnote{The computational detail can be referred to in \cite[\S 6.521.2]{gradshteyn2014table}.}.
\begin{align*}
    [P^+ f](x) :=& \frac{1}{4\pi^2} e^{im\theta} \int_{\real^2} H^{(1)}_{m}(|x||\xi|) J_m(|\xi| |y|) f(|y|) d\xi dy \\
    =& \frac{1}{2} f(x) + \frac{i}{2\pi^2} \int_{\real^2} \left|\frac{y}{x}\right|^m \frac{f(y)}{|x|^2 - |y|^2} dy \\
    [P^- f](x) :=& \frac{1}{4\pi^2} e^{im\theta} \int_{\real^2} H^{(2)}_{m}(|x||\xi|) J_m(|\xi| |y|) f(|y|) d\xi dy \\
    =& \frac{1}{2} f(x) - \frac{i}{2\pi^2} \int_{\real^2} \left|\frac{y}{x}\right|^m \frac{f(y)}{|x|^2 - |y|^2} dy. 
\end{align*}

We denote $P^\pm_N := P^\pm P_N$ to be the composition, and omit the equivariance class $m$ if it's clear. We record the following properties of $P^\pm$:
\begin{prop}[Properties of $P^{\pm}$, \cite{killip2007cubic, killip2009mass, liu2016global}]\label{inoutest}The projection $P^\pm$ defined above satisfies the following properties.
\begin{enumerate}[(1)]
  \item $P^+ + P^-$ acts as the identity on m-equivariant functions.
  \item For $|x| \gtrsim N^{-1}$ and $t \gtrsim N^{-2}$, the integral kernel obeys
    \[ \left| [P^\pm_N e^{\mp i t \Delta}](x, y) \right| \lesssim \left\{ 
    \begin{array}{ll}
      (|x||y|)^{-\frac{1}{2}} |t|^{-\frac{1}{2}} &: |y| - |x| \sim Nt  \\
      \frac{N^2}{(N|x|)^\frac{1}{2} \langle N |y| \rangle^{\frac{1}{2}}} \langle N^2 t + N|x| - N |y| \rangle^{-n} &: otherwise
    \end{array}\right. \]
    for all $n \ge 0$.
  \item For $|x| \gtrsim N^{-1}$ and $t \lesssim N^{-2}$, the integral kernel obeys
      \[ \left| [P^\pm_N e^{\mp i t \Delta}](x, y) \right| \lesssim 
          \frac{N^2}{(N|x|)^\frac{1}{2} \langle N |y| \rangle^{\frac{1}{2}}} \langle N|x| - N |y| \rangle^{-n} \]
          for all $n \ge 0$.
  \item Fix $N > 0$. Then
  \[ \| \varphi_{\gtrsim \frac{1}{N}} P^\pm_{\ge N} f \|_{L^2(\real^2)} \lesssim \| f\|_{L^2(\real^2)} \]
  with an $N-$independent constant. 
\end{enumerate}
\end{prop}

\begin{rmk}
  These result are established in \cite{killip2007cubic, killip2009mass} for radial in-out decomposition first, then \cite{liu2016global} propose the m-equivariant case above, which is similar to the radial case due to the similar asymptotic behavior of Hankel functions and Bessel functions of all orders.
\end{rmk}

\subsubsection{Almost periodic solution}

Concentration compactness argument ensures that the threshold solution enjoys certain strong concentration property in terms of charge. This is characterized by the so called almost periodicity modulo symmetries property, see \cite{tao2008minimal, liu2016global}. The following theorem and lemma will be crucial to our proof. 

\begin{thm}[Almost periodicity modulo symmetries]\label{almostpms} Let $\phi: [0, \infty) \times \real^2 \rightarrow \cpx$ be solution to (\ref{CSS}) which satisfies $\phi_0 \in L^2_m$ for $m \ge 0$, $\| \phi_0 \|_{L^2} = \| Q\|_{L^2}$, and 
\[ \| \phi \|_{L^4_{t, x}([0, \infty) \times \real^2)} = \infty. \]
Then $\phi$ is almost periodic modulo symmetries in the following sense: there exist scaling functions $N: [0, \infty) \rightarrow \real^+$ and compactness modulus function $C: \real^+ \rightarrow \real^+$, such that 
  \beq \int_{|x | \ge C(\eta)/N(t)} |\phi(t, x)|^2 dx \le \eta, \qquad 
       \int_{|\xi | \ge C(\eta)N(t)} |\hat{\phi}(t, \xi)|^2 d\xi \le \eta,  \label{apms1}\eeq
for any $t \in [0, \infty)$ and $\eta > 0$. Equivalently, the orbit $\{ N(t)^{-1}\phi(t, \frac{x}{N(t)}): t \in [0, \infty)\}$ is precompact in $L^2_x(\real^2)$. 

Moreover, $N(t)$ satisfies the local constancy property:
\beq N(t) \sim_\phi N(t_0) \label{apms2} \eeq
whenever $t, t_0 \in [0, \infty)$ and $|t-t_0 | \lesssim_\phi N(t_0)^{-2}$. And 
\beq N(t) \gtrsim_\phi N(t_0) \langle t-t_0 \rangle^{-\frac{1}{2}}, \label{apms3} \eeq
for all $t, t_0 \in [0, \infty)$.

\end{thm}

This first part of this theorem is not exactly the same as \cite[Lemma 1.7]{liu2016global}, but similarly proved using the Palais-Smale condition modulo symmetries as \cite[Proposition 2.1]{tao2008minimal}. The constancy property (\ref{apms2}) is just the same Corollary 3.6 in \cite{killip2007cubic}. The proof there is quite general. The only part relying on equation is a suitable and standard Cauchy theory from \cite{liu2016global}. (\ref{apms3}) directly follows (\ref{apms2}).

One important feature of almost periodic solution is the following Duhamel formula, where the free evolution term disappears:

\begin{lem}[Non-scattering Duhamel]\label{duhamel}
  Let $\phi$ be an almost periodic solution to (\ref{CSS}) on $[0, \infty)$ in the sense of Theorem \ref{almostpms}. Then, for all $t \in [0, \infty)$,
  \beq \phi(t) = -\lim_{T\rightarrow +\infty} i \int^T_t e^{i(t-t') \Delta} F(\phi(t')) dt' \eeq
  as a weak limit in $L^2_x$, where $F(\phi)$ is the nonlinearity of (\ref{CSS}).
\end{lem}

The proof in \cite[\S 6]{tao2008minimal} for (\ref{NLS}) merely depends on the linear evolution operator of Schr\"odinger equation, hence still applicable in our setting (\ref{CSS}).

\subsection{Reduction of the Proof for Theorem \ref{rigidinfselfdual}}\label{4.2}

In this subsection, we will see how Theorem \ref{rigidinfselfdual} reduces to the following frequency decay estimate. Therefore the remaining task in following subsections is just to prove this estimate.

\begin{prop}[Frequency decay estimate]\label{freqdecayest}
 Let $m \ge 0$,$\phi_0 \in H^1_m(\real^2)$, $\| \phi_0\|_{L^2} = \|\phi^{(m)} \|_{L^2}$ and $E[\phi] > 0$. Let $\phi = ue^{im\theta}$ be a solution to (\ref{CSS}) with initial data $\phi(0) = \phi_0$, and blows up as $t \rightarrow +\infty$. 
 Then there exists $\epsilon > 0$, such that for any dyadic number $N \ge 1$, we have
  \beq \| \varphi_{>1} P_N \phi(t) \|_{L^\infty_t L^2_x ([0, \infty) \times \real^2)} \lesssim \| \tilde{P}_N \phi_0 \|_{L^2_x} + N^{-1-\epsilon}. \eeq
\end{prop}
 
This decay estimate will imply the localization of kinetic energy Theorem \ref{stronglocal} below. Then a contradiction argument using truncated virial identity estimate will establish Theorem \ref{rigidinfselfdual}. To begin with, we state and prove Theorem \ref{stronglocal}.


\begin{thm}[Localization of kinetic energy]\label{stronglocal}
  Let $m \ge 0$, and $\phi$ be as in Theorem~\ref{freqdecayest}. 
From Theorem \ref{almostpms}, $\phi$ is almost periodic modulo symmetries. Let $N(t)$ be the corresponding scaling function, then there exists $\tilde{C}:\real^+ \rightarrow \real^+$ such that for all $\eta > 0$, 
  \[ \| \nabla \phi (t) \|_{L^2_x(\{ |x| > \tilde{C}(\eta) \langle N(t)^{-1}\rangle \})} \le \eta. \]
\end{thm}
\begin{proof}
   For any $\eta > 0 $ and $t \ge 0$, we will estimate 
   \[\| \varphi_{>R} \nabla \phi(t) \|_{L^2} < \eta \]
    with $R = 2C(\eta_1) \langle N(t)^{-1} \rangle$, where $C$ is the compactness modulus function in Theorem \ref{almostpms} and $\eta_1$ is small enough to be determined. 
    
    First we add two frequency cutoff, with $N_0$ to be determined.
       \[ \| \varphi_{>R} \nabla \phi(t) \|_2 \le \| P_{\le N_0} \varphi_{>R} \nabla \phi(t) \|_2 + \| P_{> N_0} \varphi_{>R} \nabla \phi(t) \|_2. \] 
   For the low-frequency part, we need to make the frequency truncation next to $\nabla$ and apply the charge concentration, through mismatch estimate. And finally control the main term with localization of charge (\ref{apms1}).
   \begin{align*}
  &  \| P_{\le N_0} \varphi_{>R} \nabla \phi(t) \|_2 \\
     \lesssim  &\| P_{\le N_0} \varphi_{>R} \nabla P_{\le 4 N_0} \varphi_{\le \frac{R}{2}} \phi(t) \|_2 + \| P_{\le N_0} \varphi_{>R} \nabla P_{\le 4 N_0}  \varphi_{> \frac{R}{2}} \phi(t) \|_2 \\
     &+ \sum_{N > 4N_0}  \| P_{\le N_0} \varphi_{>R} \nabla P_{N} \phi(t) \|_2  \\
     \lesssim_\phi & N_0^{-1} R^{-2} + N_0 \| \varphi_{>\frac{R}{2}} \phi(t)\|_2 + \sum_{N > 4N_0} N^{-1} R^{-2}  \\
     \lesssim_\phi & N_0^{-1}  + N_0 \eta_1
   \end{align*}
   
   For the high-frequency part, we use frequency decaying estimate Proposition \ref{freqdecayest} and almost orthogonality.
      \begin{align*}
  &   \| P_{> N_0} \varphi_{>R} \nabla \phi(t) \|_2^2 = \sum_{N > N_0}  \| P_N \varphi_{>R} \nabla \phi(t) \|_2^2\\
      = & \sum_{N > N_0}  \Big( \| P_N \varphi_{>R} P_{< \frac{N}{4}} \nabla \phi(t) \|_2^2 +   \| P_N \varphi_{>R} P_{> 4N} \nabla \phi(t) \|_2^2 \\
       &\quad+ \| P_N \varphi_{>R} P_{\frac{N}{4} \le \cdot \le 4N} \nabla \phi(t) \|_2^2\Big) \\
      \lesssim_\phi &  \sum_{N > N_0} N^{-2} R^{-4} +  \sum_{N > \frac{1}{4}N_0} \| \varphi_{>R} \nabla P_N  \phi(t) \|_2^2 \\
      \lesssim_\phi & N_0^{-2} R^{-4} + \sum_{N > \frac{N_0}{4}} \left(  \| \varphi_{>R} \nabla \tilde{P}_N \varphi_{\le \frac{R}{2}} P_N  \phi(t) \|_2^2 +  \| \varphi_{>R} \nabla \tilde{P}_N \varphi_{> \frac{R}{2}} P_N  \phi(t) \|_2^2 \right) \\
      \lesssim_\phi & N_0^{-2} R^{-4} + \sum_{N > \frac{N_0}{4}}\left[ N^{-2} R^{-4} + N^2 \left( N^{-2-2\epsilon} + \| P_N u(0) \|_2^2 \right) \right] \\
      \lesssim_\phi & N_0^{-2}  + N_0^{-2\epsilon} + \| \nabla P_{> \frac{N_0}{4}} u(0) \|_2^2  
   \end{align*}
   Choosing $N_0$ large enough, and then $\eta_1$ small enough, the conclusion is proved.
\end{proof}

Next, we are in place to prove our main theorem. 

\begin{proof}[Proof of Theorem \ref{rigidinfselfdual}]

   Assume that $E[\phi_0] > 0$. Recall the truncated virial identity $V_R(t) := \int_{\real^2} \chi_R |\phi|^2$ with $\chi_R$ is a cutoff for $|x|^2$ for $|x| \le R$, defined in Corollary \ref{trunvirialest}. Then we have (\ref{tvi2}) and a trivial bound
  \begin{align}
    |V_R(t)| &\lesssim_\phi R^2 \label{tvi3}
  \end{align}
  Besides, since $\| \phi \|_{L^2} = \| Q \|_{L^2}$ and $\phi$ blows up as $t \rightarrow +\infty$, $\phi$ is an almost periodic solution according to Theorem \ref{almostpms}, and we can apply Lemma \ref{duhamel} and Theorem \ref{stronglocal}. In the following argument, $N(t)$ and $C(\eta)$ are defined as in Theorem \ref{almostpms}. 
  
  We distinguish two cases: either $N(t)$ bounded from below (i.e. the charge, frequency and kinetic energy concentrates in a bounded area all the time) or converges to zero along a subsequence.
  
  \cu{Case 1. $\inf_{t \ge 0} N(t) > 0.$ }
  
  Let $\eta > 0$ be a small constant chosen later. From localization of charge (Theorem \ref{almostpms}) and kinetic energy (Theorem \ref{stronglocal}), there exists $R = R(\eta) = 2C(\eta) / \left(\inf_t N(t)\right)$ such that
  \beq \| \phi(t)\|_{L^2_x(|x| \ge \frac{R}{2})} +  \| \nabla \phi(t)\|_{L^2_x(|x| \ge \frac{R}{2})} \le 2 \eta \eeq
  for all $t \ge 0$. Then by Gagliardo-Nirenberg inequality,
  \beq  \| \varphi_{\ge \frac{R}{2}} \phi(t) \|_{L_x^{4}} \lesssim_\phi \eta \eeq
  provided $R$ is chosen sufficiently large depending on $\eta$. 
  
  Hence, taking $\eta$ small enough depending on $E[\phi_0]$ and $R$ correspondingly large, the residual in (\ref{tvi2}) will be small enough. We obtain
  \[ V_R''(t) \ge 8 E[\phi_0] > 0 \]
  for all $t \in [0, \infty)$, thus contradicting (\ref{tvi3}) for $t$ large enough.
  
  \cu{Case 2.} $\liminf_{t \rightarrow \infty} N(t) = 0.$ 
  
  In this case, we are not able to choose a fixed $R$ to guarantee uniform concentration. 
   So the idea here is to consider the contradiction on the  speed of divergence, due to the a priori bound (\ref{apms3}). 
  
  Choose the time sequence $t_n \nearrow \infty$ such that $N(t_n) \searrow 0$ and 
  \beq N(t_n) = \min_{0 \le t \le t_n} N(t). \eeq
  Again, let $\eta > 0$ be a small constant and $R_n = 2C(\eta) / N(t_n)$, then 
  \beq \| \phi(t)\|_{L^2_x(|x| \ge \frac{R_n}{2})} +  \| \nabla \phi(t)\|_{L^2_x(|x| \ge \frac{R_n}{2})} + \| \phi(t) \|_{L^{4/d+2}_x(|x| \ge \frac{R_n}{2})} \lesssim_\phi \eta \eeq
  for all $t \in [0, t_n]$. Thus similarly, an $\eta$ small enough depending on $E[\phi_0]$ (but independent of $n$) implies
    \beq V_{R_n}''(t) \ge 8 E[\phi_0] > 0 \label{bulabula}\eeq
  for all $t \in [0, t_n)$.
  
  On the other hand, from the differential inequality for truncated virial (\ref{virialode}) and (\ref{tvi3}), we have
  \beq  |V_{R_n}'(t)| \lesssim_{E(\phi_0)} (V_{R_n}(t))^{1/2} \lesssim_\phi R_n \eeq
  for all $t \in [0, t_n]$. Thus using Fundamental Theorem of Calculus and (\ref{bulabula}), we obtain
  \[ E[\phi_0] t_n \le |V_{R_n}'(0)| +|V_{R_n}'(t_n)| \lesssim_\phi R_n \lesssim_\phi N(t_n)^{-1} \lesssim_\phi t_n^{1/2}. \]
  Letting $n \rightarrow \infty$, we reach a contradiction since $t_n \rightarrow \infty$.

  This finishes the proof of Theorem \ref{rigidinfselfdual}.
  
\end{proof}

\subsection{Weak localization of kinetic energy}\label{4.3}

Now the only goal for the rest of this paper is to prove Proposition \ref{freqdecayest}. To achieve that, we need quantities controlling the solution uniformly for all time, and just the charge conservation is not enough in dimension 2. So we will prove the following uniform $\dot{H}^1$ boundedness property in this subsection, as a preparation for proving Proposition \ref{freqdecayest}.

\begin{prop}[Weak localization of kinetic energy]\label{weaklocal}
  For $m \ge 0$, any $\phi \in H^1_m$, $\|\phi\|_{L^2} = \|\phi^{(m)} \|_{L^2}$ ($\phi^{(m)}$ be the m-equivariant soliton) and $E[\phi] > 0$, then for all $c > 0$,
  \beq \| \varphi_{>c} \nabla \phi \|_{L^2} \lesssim_{c, E[\phi], m} 1 .\eeq
\end{prop}

\begin{rmk}
  Due to conservation of energy, this Theorem indicates that $\nabla \phi(t)$ cannot be large in region away from the origin. So it can also be regarded as a weaker localization of kinetic energy, compared with the strong one Theorem \ref{stronglocal}.
\end{rmk}

To begin with, we need some spectral analysis of linearized equation around soliton.

\subsubsection{Linearization of (\ref{CSS}) at the soliton}

  In this subsubsection, we  consider the linearization of (\ref{CSS}) at the m-equivariant soliton $\phi^{(m)}=e^{im\theta}Q$, to see that the linearized operator can be written as a self-dual form
  \beq \mathcal{L}_Q = L_Q^* L_Q. \eeq
Namely, if we consider the $u$-evolution formulation (\ref{uCSS}) for (\ref{CSS}),  and denote $u = Q + \epsilon$,  then (\ref{uCSS}) is equivalent to 
\[ i \partial_t \epsilon - L_Q^* L_Q \epsilon = \text{(h.o.t.)}. \]
And then we record its spectral properties, the analysis of which largely depends this self-dual structure. Most of these results appear in \cite{kim2019css}, so we omit their proof.

We begin with linearization of Bogomol'nyi operator
  \begin{align*}
    \bm{D}_+^{(u)} &:= \partial_r - \frac{1}{r} (m + A_\theta[u]), \\
    \bm{D}_+^{(u)*} &:= -\partial_r - \frac{1}{r} (1+m + A_\theta[u]).
  \end{align*}
  Equation (\ref{uCSS}) can be written as a Hamiltonian equation
  \[ \partial_t u = - i \frac{\delta E}{\delta u} = -i \frac{\delta}{\delta u}\left( \frac{1}{2} \int |\bm{D}^{(u)}_+ u |^2 \right). \]
Corresponding to the quadrature structure of basic nonlinearity $A_\theta[u]$, we define the multiplication operator and its adjoint.\footnote{We remark that the operator as well as $A_\theta$ is only $\real$-linear, rather than $\cpx$-linear. So all the adjoint afterwards are viewed as in $\real$-Hilbert space $L^2(\real^2; \cpx)$ equipped with the inner product $(u, v)_r = \int \text{Re}(u\bar{v})$.}
\begin{align*} B_f g :=& \frac{1}{r} \int_0^r \text{Re} (\bar{f} g) r' dr'. \\
B_f^* g =& f \int^\infty_r (\text{Re} g) dr'. 
\end{align*}
With that, we can represent those nonlinearity appearing in (\ref{CSS})
\begin{align*}
  A_\theta[u]u &= -\frac{1}{2}r (B_u u)u,\\
  A_0[u]u &= - B^*_u\left[\frac{m}{r} |u|^2 - \frac{1}{2} |u|^2 B_u u \right].
\end{align*}
Then assuming the decomposition at arbitrary profile $\omega$
\[ u = \omega + \epsilon,\]
we can further decompose the operator
\begin{align*}
  \bm{D}^{(u)}_+ &= \bm{D}^{(\omega)}_+ + (B_\omega \epsilon) + \frac{1}{2} (B_\epsilon \epsilon),\\
  \bm{D}^{(u)*}_+ &= \bm{D}^{(\omega)*}_+ + (B_\omega \epsilon) + \frac{1}{2} (B_\epsilon \epsilon).
\end{align*}
And hence 
\begin{align}
  \bm{D}^{(u)}_+ u = \bm{D}^{(\omega)}_+ u + L_\omega \epsilon + N_\omega [\epsilon], \label{lindecomp}
\end{align}
where the linear part $L_\omega$ and nonlinear part $N_\omega[\epsilon]$ are respectively
\begin{align*}
  L_\omega &:= \bm{D}^{(\omega)}_+ + \omega B_\omega, \\
  N_\omega[\epsilon] &:= \epsilon B_\omega \epsilon + \frac{1}{2} \omega B_\epsilon \epsilon + \frac{1}{2}\epsilon B_\epsilon \epsilon.
\end{align*}
And the real adjoint of $L_\omega$ is 
\[ L^*_\omega f = \bm{D}_+^{(\omega)*} f + B^*_\omega (\bar{\omega} f). \]
In particular, when $\omega = Q$, using the self-dual relation $D^{(Q)}_+ Q = 0$, we have
\begin{lem}[(\ref{CSS}) in the self-dual form, \cite{kim2019css}]\label{CSSlinear}
    Using the previous notation, the self-dual (\ref{CSS}) under equivariant assumption is equivalent to 
    \beq i \partial_t u = L_u^* \bm{D}_+^{(u)} u. \eeq
    Moreover, suppose $u = Q + \epsilon$, it's equivalent to the linearized equation
  \beq
 \begin{split}
    i \partial_t \epsilon - \mathcal{L}_Q \epsilon =& L^*_Q N_Q[\epsilon] + \left[ (B_Q \epsilon) + B^*_Q [\bar{\epsilon} \cdot] + B^*_\epsilon [Q \cdot] \right]\left[L_Q \epsilon + N_Q [\epsilon]\right] \\
    &+ \left[\frac{1}{2} (B_\epsilon \epsilon) + B^*_\epsilon [\bar{\epsilon} \cdot] \right]\left[L_Q \epsilon + N_Q [\epsilon]\right],
  \end{split}
\eeq
where the linearized operator is 
 \beq \mathcal{L}_Q = L_Q^* L_Q. \eeq
\end{lem}

Differentiating symmetries of (\ref{CSS}) at the static soliton $Q$, we obtain explicit algebraic identities satisfied by $\mathcal{L}_Q$. From phase and scaling symmetries, we have
\begin{align*}
  \mathcal{L}_Q [i Q] = 0, \qquad \mathcal{L}_Q [\Lambda Q] = 0.
\end{align*}
where $\Lambda$ is the generator for $L^2$-scaling
\beq \Lambda f:= \frac{d}{d\lambda}\Big|_{\lambda = 1} \lambda f(\lambda \cdot) = [1 + r\partial_r] f. \eeq
One can indeed see that $iQ$ and $\Lambda Q$ span the kernel of $L_Q$. And the coercivity of $\mathcal{L}_Q$ follows from the factorization $\mathcal{L}_Q = L^*_Q L_Q$. 
\begin{lem}[Kernel of $L_Q$, \cite{kim2019css}]
  If $f(r) e^{im\theta}$ is a smooth m-equivariant function such that $L_Q f = 0$, then $f \in \text{span}_\real \{ iQ, \Lambda Q\}$. 
\end{lem}

\begin{lem}[Coercivity of $\mathcal{L}_Q$, \cite{kim2019css}]\label{coercivity}
  Let $m \ge 1$, we have\footnote{Recall the notation in \S \ref{sec2.1} that $\| u \|_{\dot{H}^1_m} :=\| u e^{im\theta}\|_{\dot{H}^1}$.}
  \begin{align}
   \| L_Q u \|_{L^2} &\lesssim \| u\|_{\dot{H}^1_m},\qquad \forall\, ue^{im\theta} \in \dot{H}^1_m, \label{coer1}\\
   \| L_Q u \|_{L^2} &\gtrsim \| u\|_{\dot{H}^1_m}, \qquad \forall\, ue^{im\theta} \in \dot{H}^1_m, \,\, (u, iQ)_r = (u, \Lambda Q)_r = 0. \label{coer2}
   \end{align}
   And in case of $m=0$, (\ref{coer1}) and (\ref{coer2}) are still true if we replace $\dot{H}^1_m$ by
   \beq \|u \|_{\dot{\mathcal{H}}_0}^2 := \|\partial_r u\|_{L^2}^2 + \| (1+r)^{-1} u\|_{L^2}^2. \label{coer3}\eeq
\end{lem}

\begin{rmk}When $m \ge 1$, we have $iQ, \Lambda Q \in (\dot{H}^1_m)^*$, so the inner products in (\ref{coer2}) are defined naturally, while for $m = 0$, $iQ, \Lambda Q \notin (\dot{\mathcal{H}}_0)^*$ makes these inner products risky. Despite this, we use this lemma later only for $u \in H^1$ (see Lemma \ref{rigidquan}). Then it's not a problem since $Q \in L^2$ for all $m \ge 0$.\end{rmk}


With all the information on  $\mathcal{L}_Q$, we are able to compute the leading term in $E[Q + \cdot]$.
\begin{lem}\label{taylor}
    For $m \ge 1$, $\epsilon e^{im\theta} \in H^1_m$, and $Qe^{im\theta}$ be the static soliton, we have
    \beq \left| 2E\left[(Q + \epsilon)e^{im\theta}\right] - \| L_Q \epsilon\|_{L^2}^2\right| \lesssim_m \| \epsilon \|_{\dot{H}^1_m}^2 \left( \sum_{k=1}^4 \| \epsilon\|_{H^1_m}^k \right).\label{qcn} \eeq
    And for $m = 0$, (\ref{qcn}) still holds if we replace $\dot{H}^1_0$ by $\dot{\mathcal{H}}_0$ defined in (\ref{coer3}).
\end{lem}

\begin{proof}[Proof of Lemma \ref{taylor}]
   Recall the decomposition (\ref{lindecomp}) and note that $\bm{D}^{(Q)}_+ Q = 0$, we have
    \begin{align*}
        &2E\left[ (Q+\epsilon)e^{im\theta} \right] -\|L_Q \epsilon\|_{L^2}^2= \int |\bm{D}^{(Q+\epsilon)}_+ (Q+\epsilon) |^2 - |L_Q \epsilon|^2 \\
        =& \int |\bm{D}^{(Q)}_+ Q + L_Q \epsilon + N_Q[\epsilon] |^2 - |L_Q \epsilon|^2 
        = \int |L_Q \epsilon + N_Q[\epsilon] |^2 - |L_Q \epsilon|^2 \\
        = &\int 2 \text{Re}(L_Q \epsilon \cdot \overline{N_Q [\epsilon]}) + | N_Q [\epsilon]|^2 
    \end{align*}
    Thus, (\ref{qcn}) follows from the following $L^2$ estimates for $L_Q \epsilon$ and $N_Q [\epsilon]$ (for $m = 0$ again substitute $\dot{H}^1_0$ by $\dot{\mathcal{H}}_0$).
    \begin{align}
       \| L_Q \epsilon\|_{L^2}&\lesssim \|\epsilon \|_{\dot{H}^1_m}, \label{qcn3}\\
       \| N_Q [\epsilon] \|_{L^2} & \lesssim  \|\epsilon \|_{\dot{H}^1_m} \left(  \|\epsilon \|_{H^1_m} +  \|\epsilon \|_{H^1_m}^2 \right).\label{qcn4}
    \end{align}
    (\ref{qcn3}) comes from Lemma \ref{coercivity}. Next we prove (\ref{qcn4}). Recall that 
    \[ N_Q[\epsilon] = \epsilon B_Q \epsilon + \frac{1}{2}Q B_\epsilon \epsilon + \frac{1}{2} \epsilon B_\epsilon \epsilon, \]
    where 
    \[ B_f g = \frac{1}{r} \int^r_0 \text{Re}(\bar{f} g) r' dr'. \] 
    We distinguish two cases. 
    
    \cu{Case 1. $m \ge 1$.} 
    
    Now that $\frac{1}{r} \epsilon$ is bounded in $L^2$, we have the following estimate
    \[ \| f_1 B_{f_2} f_3 \|_{L^2} \lesssim  \| \frac{1}{r}f_\alpha \|_{L^2}  \| f_\beta \|_{L^2}  \| f_\gamma \|_{L^2},\quad \text{for}\,\, \{ \alpha, \beta, \gamma \} = \{ 1, 2, 3 \} \]
    by putting $\frac{1}{r}$ on $f_1$ or inside the integral of $B_{f_2} f_3$ to be $\frac{1}{r'}$ then applying Cauchy-Schwartz inequality. Using this estimate and take $\frac{1}{r}$ onto $\epsilon$, (\ref{qcn4}) follows.
    
    \cu{Case 2. $m = 0$.} 
    
    In this case, we need to be more careful. Recall that $Q = \sqrt{8} \frac{1}{1+r^2}$,
     we see
    \[ |B_Q \epsilon(r)| \lesssim \frac{1}{r} \int_0^r \frac{1}{1+(r')^2} \| r^{\frac{1}{2}} \epsilon \|_{L^\infty} \frac{1}{(r')^\frac{1}{2}} r'dr' \lesssim \min \{ r^{\frac{1}{2}}, r^{-1} \} \|r^{\frac{1}{2}} \epsilon \|_{L^\infty}.\]
    By the Strauss' estimate in $\real^2$ 
    \beq |f(r)| \lesssim \| f\|_{\dot{H}^1(\{|x| \ge r \})}^{\frac{1}{2}} \|f\|_{L^2(\{|x| \ge r \})}^{\frac{1}{2}} r^{-\frac{1}{2}}, \label{strauss1} \eeq
    we have
    \begin{align*}
     \| \epsilon B_Q \epsilon \|_{L^2} &\le \|r^{\frac{1}{2}}  \epsilon \|_{L^\infty} \| r^{-\frac{1}{2}} B_Q \epsilon \|_{L^2} \\
    &\lesssim  \|r^{\frac{1}{2}} \epsilon \|_{L^\infty}^2 \left\| \min\{1, r^{-\frac{3}{2}}\} \right\|_{L^2} \lesssim  \| \epsilon \|_{\dot{H}^1} \| \epsilon \|_{L^2}.
    \end{align*}
    And for $B_\epsilon \epsilon$,
    \[ |B_\epsilon \epsilon(r)| \lesssim \frac{1}{r} \int_0^r dr'  \| r^{\frac{1}{2}} \epsilon \|_{L^\infty}^2 \lesssim  \| r^{\frac{1}{2}} \epsilon \|_{L^\infty}^2, \]
    then 
    \begin{align*} \|Q B_\epsilon \epsilon\|_{L^2} +  \|\epsilon B_\epsilon \epsilon\|_{L^2} &\lesssim (\|Q\|_{L^2} + \|\epsilon\|_{L^2} ) \| r^{\frac{1}{2}} \epsilon \|_{L^\infty}^2 \\
     &\lesssim \| \epsilon \|_{\dot{H}^1} (\| \epsilon \|_{L^2} +\| \epsilon \|_{L^2}^2 ). \end{align*}
    Thus (\ref{qcn4}) holds for $m \ge 0$.

\end{proof}

\subsubsection{Modulation analysis}

\begin{lem}[Rigidity of the ground state, quantitative version]\label{rigidquan}
  Let $m \ge 0$, and $\phi^{(m)} = Qe^{im\theta}$ be the corresponding soliton. There exists constants $\eta > 0$, $C > 1,\,K>0$ such that the following be true.
  
  Let $\phi = u e^{im\theta} \in H^1_m$ be such that
  \[ \| \phi \|_{L^2} =\| \phi^{(m)} \|_{L^2}, \qquad  \| \nabla \phi \|_{L^2} = \| \nabla \phi^{(m)} \|_{L^2},\]
  and
  \beq  E[\phi] \le \eta.  \label{rigquan1}\eeq
  Then there exist $\gamma_0 = \gamma_0 (\phi) \in \real$, $\lambda_0 = \lambda_0 (\phi) > 0$ with
  \beq \frac{1}{C} \le \lambda_0 \le C \label{rigquan2}\eeq
  such that
  \[ \epsilon = e^{i\gamma_0} \lambda_0 u(\lambda_0 \cdot) - Q \]
  satisfies the following:
  \begin{enumerate}[(1)]
    \item The orthogonality condition
      \beq (\text{Re}(\epsilon), \Lambda Q)_r = (\text{Im}(\epsilon), Q)_r = 0.\label{rigquan3}\eeq
    \item The bound of $\dot{H}^1_m$ norm (substituted by $\dot{\mathcal{H}}_0$ in (\ref{coer3}) for $m=0$ case)
      \beq \| \epsilon\|_{\dot{H}^1_m} \le K \sqrt{E[\phi]}. \label{rigquan4} \eeq
  \end{enumerate}
\end{lem}

\begin{rmk}
  Compared with the rigidity in Corollary \ref{rigidqual}, here we have stronger quantitative estimate of the residual. But it's only a $\dot{H}^1_m$ control, which essentially comes from the $\dot{H}^1_m$ coercivity of $\mathcal{L}_Q$ (Lemma \ref{coercivity}).
\end{rmk}

\begin{proof}
\cu{Step 1. Modulation method.}
 
 We show that (\ref{rigquan3}) holds. Define the $H^1$ neighborhood of $\phi^{(m)}$
 \[ U_\alpha := \{ \phi \in H^1_m(\real^2) : \|\phi - \phi^{(m)} \|_{H^1_m} < \alpha\}. \]
 For any $\gamma \in \real,\,\lambda > 0,\,\phi = ue^{im\theta} \in H^1_m$, define 
 \beq \epsilon_{\lambda, \gamma} := e^{i\gamma}\lambda u (\lambda \cdot) - Q. \eeq
 We claim that there exists $\alpha_0 > 0$ and a unique map $(\lambda, \gamma) : U_{\alpha_0} \rightarrow \real^+ \times \real$ satisfying:
 \begin{align}
   (\text{Re}(\epsilon_{\lambda, \gamma}), \Lambda Q)_r = (\text{Im}(\epsilon_{\lambda, \gamma}), Q)_r = 0.
 \end{align}
Furthermore, there exists a constant $K_1 > 0$ such that for $0 < \alpha < \alpha_0$, $\phi \in U_\alpha$, then
\beq    \| \epsilon_{\lambda, \gamma} \|_{H^1_m} + |\lambda - 1| + | \gamma| \le K_1\alpha.\label{boundmm}  \eeq

Consider the functional
\[ \rho_1(\phi, \lambda, \gamma) :=  (\text{Re}(\epsilon_{\lambda, \gamma}), \Lambda Q)_r,\qquad \rho_2 (\phi, \lambda, \gamma) := (\text{Im}(\epsilon_{\lambda, \gamma}), Q)_r, \]
and note that $\rho_1 (Q, 1, 0 ) = \rho_2 (Q,1 , 0)= 0$,  we only need to show
\beq \frac{\partial (\rho_1, \rho_2)}{\partial (\lambda, \gamma)}\Bigg|_{(Q, 1, 0)}  \label{jacob} \eeq
non-degenerate, then apply the implicit function theorem for Banach space. Note that
\beq \frac{\partial \epsilon_{\lambda, \gamma}}{\partial \lambda}\Big|_{(1, 0)} = \Lambda u,\qquad \frac{\partial \epsilon_{\lambda, \gamma}}{\partial \gamma}\Big|_{(1, 0)} = i u.\eeq
Thus 
\[ (\ref{jacob}) = \left(
\begin{array}{cc}
  \| \Lambda Q \|_{L^2}^2 & 0 \\
  0 & \| Q \|_{L^2}^2
\end{array}\right)
\]
is non-degenerate. So the claim holds, which implies that (\ref{rigquan3}) holds if $\phi \in U_{\alpha_0}$.

Finally, using Corollary \ref{rigidqual}, for sufficiently small $\eta$, (\ref{rigquan1}) implies $\phi \in U_{\alpha_0}$. We've finished the proof of (\ref{rigquan3}). Also (\ref{rigquan2}) comes immediately after (\ref{boundmm}).

\cu{Step 2. Quantitative control of $\dot{H}^1_m$ norm}

By Step 1, we already know $\epsilon$ satisfies the orthogonality condition of Lemma \ref{coercivity}, so we can use $\|L_Q \epsilon \|_{L^2}$ to control the $\dot{H}^1_m$ norm by (\ref{coer2}). Recalling Lemma \ref{taylor}, this quantity is essentially $\sqrt{E(\phi)}$ with higher order errors. Combined with the scaling property $\lambda_0^2 E[\phi] = E\left[(Q+\epsilon)e^{im\theta}\right]$, for $m \ge 1$, we have
\begin{align}
 \| \epsilon \|_{\dot{H}^1_m}^2 \le C_m \| L_Q \epsilon \|_{L^2}^2 \le  2 C_m \lambda_0^2 E\left(\phi\right) + C_m C_m' \| \epsilon \|_{\dot{H}^1_m}^2 \left( \sum_{k=1}^4 \| \epsilon\|_{H^1_m}^k \right). \label{qcn2}
\end{align}
 From (\ref{boundmm}), if we take $\eta$ sufficiently small so that $\| \epsilon\|_{H^1_m}$ small enough and satisfies
 \[ \left( \sum_{k=1}^4 \| \epsilon\|_{H^1_m}^k \right) C_m C_m'  \le \frac{1}{2},\]
 then left hand side of (\ref{qcn2}) can absorb the last term on the right, which complete the proof for $m \ge 1$. As for $m = 0$ case, just replace $\dot{H}^1_0$-norm by $\dot{\mathcal{H}}_0$ and the above estimates still hold.
 
\end{proof}

\subsubsection{Proof of Proposition \ref{weaklocal}}

We can apply the modulation analysis to get the following non-sharp decomposition Proposition \ref{nonsharpdecomp}, which provides a useful bound for the distance to the soliton family. 

\begin{prop}[Non-sharp decomposition of $H^1_m$ function with threshold charge]\label{nonsharpdecomp}
   For $m \ge 0$, there exists $C_1, C_2 > 0$ such that: for any $\phi = ue^{im\theta} \in H^1_m$ with $\|\phi\|_{L^2} = \|\phi^{(m)} \|_{L^2}$ ($\phi^{(m)}$ be the m-equivariant soliton), there exists $\theta_0 = \theta_0(\phi) \in \real$, $\lambda = \lambda (\phi)$, $\epsilon e^{im\theta} = \epsilon(\phi) e^{im\theta} \in H^1_m$ for which we have
   \[ u = \lambda e^{i\theta_0} Q(\lambda \cdot)+ \epsilon, \]
   where
   \[ \frac{1}{C_2} \cdot \frac{\| \nabla \phi \|_{L^2}}{\|\nabla \phi^{(m)} \|_{L^2}} \le \lambda \le C_2 \cdot \frac{\| \nabla \phi \|_{L^2}}{\|\nabla \phi^{(m)} \|_{L^2}},\qquad \text{if} \,\, \| \nabla \phi \|_{L^2}^2 \ge C_1 E[\phi], \]
   and 
   \[ \lambda = 1 ,\qquad \text{if} \,\, \| \nabla \phi \|_{L^2}^2 \le C_1 E[\phi]. \]
   The term $\epsilon$ satisfies the bound
   \beq \| \epsilon \|_{\dot{H}^1_m} \lesssim \sqrt{E[\phi]} + 1. \label{nonsharp2}\eeq 
\end{prop}

\begin{proof}[Proof of Proposition \ref{nonsharpdecomp}]
  Let $\phi = ue^{im\theta}\in H^1_m$ and $\|\phi \|_{L^2} = \| \phi^{(m)}\|_{L^2} $. Rescale by $\mu = \frac{\| \phi^{(m)} \|_{\dot{H}^1_m}}{ \|\phi \|_{\dot{H}^1_m}} $,  $\tilde{\phi} := \mu \phi(\mu \cdot)$. Then 
  \[ \| \tilde{\phi}\|_L^2 =\| \phi^{(m)} \|_{L^2},\quad \| \tilde{\phi} \|_{\dot{H}^1_m} = \| \phi^{(m)} \|_{\dot{H}^1_m},\quad E[\tilde{\phi}] = \frac{ \| \phi^{(m)} \|_{\dot{H}^1_m}^2 }{\|\phi \|_{\dot{H}^1_m}^2} E[\phi]. \]
  Take  $C_1 :=\frac{\| \phi^{(m)} \|_{\dot{H}^1_m}^2}{\eta}$ with $\eta$ as in Lemma \ref{rigidquan}. We distinguish two cases.
  
  \cu{Case 1. $\| \phi \|_{\dot{H}^1_m}^2  \ge C_1 E[\phi]$.}
  
 Note that the condition is exactly the smallness condition (\ref{rigquan1}) for rigidity of $\tilde{\phi}$. We can apply Lemma \ref{rigidquan} for $\tilde{\phi}$. Thus there exists $\tilde{\gamma}_0 \in \real$ and $\tilde{\lambda}_0 \in [\frac{1}{C}, C]$ and 
  \begin{align}
   \tilde{\epsilon} := e^{i\tilde{\gamma}_0} \tilde{\lambda}_0 \tilde{u} (\tilde{\lambda}_0 \cdot) - Q, \nonumber \\
   \| \tilde{\epsilon}\|_{\dot{H}^1_m} \le K \sqrt{E[\tilde{\phi}]}.  \label{energytilde}
    \end{align}
   Then after rescaling, we see 
   \[ u = e^{i\theta_0} \lambda Q(\lambda \cdot) + \epsilon \]
   where $\theta_0 = - \tilde{\gamma}_0$,  
   $\lambda = (\tilde{\lambda}_0 \mu )^{-1} \in \left[ \frac{1}{C} \frac{\| \phi \|_{\dot{H}^1_m}}{ \|\phi^{(m}) \|_{\dot{H}^1_m}}, C \frac{\| \phi \|_{\dot{H}^1_m}}{ \|\phi^{(m}) \|_{\dot{H}^1_m}}\right] $
   and 
   $\epsilon = \lambda \tilde{\epsilon}(\lambda \cdot) $. Since $\lambda \sim \mu^{-1}$, by rescaling of (\ref{energytilde}) we immediately obtain (\ref{nonsharp2}).
   
   \cu{Case 2. $\| \phi \|_{\dot{H}^1_m}^2 \le C_1 E[\phi] $.}
   
   We just set $\theta_0 = 0, \, \lambda = 1$ and $\epsilon = u - Q$. Then (\ref{nonsharp2}) comes from a rough bound
   \[ \| \epsilon \|_{\dot{H}^1_m} \le  \| Q \|_{\dot{H}^1_m}  +  \| u \|_{\dot{H}^1_m} \le  \| Q \|_{\dot{H}^1_m}  +  \| \phi \|_{\dot{H}^1_m}  \lesssim_m 1 + \sqrt{E[\phi]}. \]
   
\end{proof}

Finally, we conclude 

\begin{proof}[Proof of Proposition \ref{weaklocal}]
  Apply Proposition \ref{nonsharpdecomp}. Then the bound is obvious when $\| \nabla \phi\|_{L^2}^2 \le C_1 E[\phi]$.
  
  If $\| \nabla \phi\|_{L^2}^2 \ge C_1 E[\phi]$, 
  \begin{align*} 
  \| \varphi_{>c} \nabla \phi \|_{L^2} &\lesssim \| \varphi_{>c} \partial_r \phi \|_{L^2} + \| \varphi_{>c} \frac{1}{r} \phi \|_{L^2} \\
  &\lesssim \| \varphi_{>c} \partial_r (\lambda Q(\lambda \cdot))\|_{L^2} + \| \varphi_{>c} \partial_r \epsilon \|_{L^2} + \| \phi \|_{L^2 }. 
   \end{align*}
   The third term is bound by threshold charge $\| \phi \|_{L^2} = \|Q \|_{L^2}$, and the second by  (\ref{nonsharp2}). For the first term, recall that $Q = C_m \frac{r^{m}}{1+ r^{2(m+1)}}$ have good decay property away from the origin.
  \begin{align*}
    \| \lambda Q(\lambda \cdot) \|_{\dot{H}^1_m (|x| \ge c)}^2& \lesssim_m \int_c^\infty \lambda^2 \left(\lambda^2 |\partial_r Q(\lambda r)|^2 + \frac{1}{r^2} Q(\lambda r)^2\right) rdr\\
   & \lesssim_m \int_c^\infty \lambda^4 \left(\frac{1}{\lambda r}\right)^{2m + 6} rdr   = \lambda^2 \int^\infty_{\lambda c} s^{-2m-5} dr \\
   & \lesssim_m \lambda^{-2m -2} c^{-2m-4} \lesssim_{m,E[\phi],c} 1
  \end{align*}
  The last inequality follows from $\lambda \sim_m \| \phi \|_{\dot{H}^1_m} \gtrsim E[\phi] > 0$ in this case.

\end{proof}

\subsection{Proof of Proposition \ref{freqdecayest}}\label{4.4}

To establish this proposition, we first summarize the estimate we will use for the nonlinearity. 

\begin{lem}[Estimate for nonlinearity]\label{nonlinearest}
    Let $\phi$ be such as in Proposition \ref{freqdecayest} and $F(\phi)$ be the nonlinearity (\ref{nl}) of (\ref{CSS}), we have the following estimate
    \begin{align}
      \| \varphi_{> \frac{1}{4}} F(\phi) \|_{L^\infty_t L^2_x([0, \infty) \times \real^2)}& \lesssim_\phi 1 \label{nonest1}\\
      \| \varphi_{> \frac{1}{4}} \partial_r (F(\phi)) \|_{L^\infty_t L^2_x([0, \infty) \times \real^2)}& \lesssim_\phi 1  \label{nonest2}\\
      \| \varphi_{> T} F(\phi) \|_{L^\infty_t L^1_x([0, \infty) \times \real^2)}& \lesssim_\phi T^{-\frac{1}{2}}  \label{nonest3}\\
      \| \varphi_{> T} \partial_r (F(\phi)) \|_{L^\infty_t L^1_x([0, \infty) \times \real^2)}& \lesssim_\phi T^{-\frac{1}{2}} \label{nonest4}
    \end{align}
\end{lem}

 Next we prove Proposition \ref{freqdecayest} with these estimates, and defer their proof till the end of this subsection.
\begin{proof}[Proof of Proposition \ref{freqdecayest}]
  We begin by projecting $\phi$ onto incoming and outgoing waves, and use Duhamel's formula backward Lemma \ref{duhamel} (since $\phi$ is almost periodic from Theorem \ref{almostpms}) and forward in time, respectively.
      \begin{align}
   & \varphi_{>1} P_N \phi(t) = \varphi_{>1} P^+_N \phi(t) + \varphi_{>1} P^-_N \phi(t) \nonumber\\
    =  &    \varphi_{>1} P^-_N e^{it \Delta} \phi_0 \label{decomp1-1}\\
    & + i \int^\infty_0  \varphi_{>1} P^+_N e^{-i \tau \Delta} F(\phi(t + \tau)) d\tau \label{decomp1-2}\\
   &- i \int^t_0  \varphi_{>1} P^-_N e^{i \tau \Delta} F(\phi(t - \tau)) d\tau\label{decomp1-3}
  \end{align}
  The last two integral should be understood in the weak $L^2_x$ sense, for which our estimate still valid thanks to Fatou's property. Since $\varphi_{>1} P_N^-$ is a bounded operator from $L^2$ to $L^2$ by Proposition \ref{inoutest} (4),  the first term is controlled by Strichartz estimate
  \[ \|(\ref{decomp1-1})\|_{L^2_x} \lesssim \| \tilde{P}_N \phi_0 \|_{L^2}. \]
  Now, we only give the details of estimate of (\ref{decomp1-2}), and (\ref{decomp1-3}) will be done in the similarly way thus omitted. We start by decomposing\footnote{We remark that the partitioning point $\frac{N\tau}{2}$ in (\ref{decomp2-3}) and (\ref{decomp2-4}) may be modified into $\frac{N\tau}{C}$ with $C \ll 1$, in order that the stationary phase in Proposition \ref{inoutest} (2) will not be touched. Here we take $C = 2$ for simplicity.}
    \begin{align}
    i\varphi_{>1} \int^\infty_0 & P^+_N e^{-i \tau \Delta} F(\phi(t + \tau)) d\tau \nonumber   \\
     =& i\varphi_{>1} \int^{N^{-1}}_0P^+_N e^{-i \tau \Delta} \varphi_{\le \frac{1}{2}} F(\phi(t + \tau)) d\tau \label{decomp2-1}\\
    &+  i\varphi_{>1} \int^{N^{-1}}_0 P^+_N e^{-i \tau \Delta} \varphi_{> \frac{1}{2}} F(\phi(t + \tau)) d\tau  \label{decomp2-2} \\
    &+  i\varphi_{>1}\int_{N^{-1}}^\infty P^+_N e^{-i \tau \Delta} \varphi_{\le \frac{N\tau}{2}} F(\phi(t + \tau)) d\tau\label{decomp2-3}\\ 
    &+ i\varphi_{>1} \int_{N^{-1}}^\infty P^+_N e^{-i \tau \Delta} \varphi_{> \frac{N\tau}{2}} F(\phi(t + \tau)) d\tau  \label{decomp2-4}
  \end{align}

  The remaining part of this proof is devoted to estimating these four pieces.
  
  \cu{Estimate of (\ref{decomp2-1}).}   
  
  This is the tail term considering the decaying estimate of $P^+_N e^{-it\Delta}$ as in Proposition \ref{inoutest}. To make use of the $\dot{H}^1$ control away from the origin, we use the equation $(i\partial_t + \Delta)\phi = F(\phi)$ to replace $F(\phi)$ by $(i\partial_t + \Delta) \phi$. Note that
  \begin{align}
   (\ref{decomp2-1})=& i\varphi_{>1}P^+_N \int^{\frac{1}{N}}_0 e^{-i\tau \Delta} (\varphi_{\le \frac{1}{2}} (i\partial_\tau + \Delta)\phi(t + \tau)) d\tau \nonumber \\
  =& -\varphi_{>1}P^+_N e^{-i\frac{1}{N}\Delta} (\varphi_{\le \frac{1}{2}} \phi(t+\frac{1}{N})) \\
  &+ \varphi_{>1}P^+_N (\varphi_{\le \frac{1}{2}} \phi(t)) \\
  &- i\varphi_{>1}P^+_N \int^{\frac{1}{N}}_0 e^{-i\tau \Delta} \phi(t+\tau) \Delta \varphi_{\le \frac{1}{2}} d\tau\\
  &- 2i\varphi_{>1}P^+_N \int^{\frac{1}{N}}_0 e^{-i\tau \Delta} \nabla \phi(t+\tau) \cdot \nabla \varphi_{\le \frac{1}{2}} d\tau \label{decomp2-1-1}
  \end{align}
  where the second equality used commutator $[\partial_\tau, e^{-i\tau \Delta}] = -i\Delta$ and integration by parts on $\tau$.
  These four terms are going to be estimated in the same manner, so we only estimate (\ref{decomp2-1-1}) for instance. By Proposition \ref{inoutest}, the kernel obeys the estimate
  \begin{align*}
      |[\varphi_{> 1} P^+_N e^{-i\tau\Delta} \kappa_{\le \frac{1}{2}}] (x, y) | &\lesssim N^2 \langle N|x| - N|y| \rangle^{-n} \kappa_{|x| > 1} \kappa_{|y| \le \frac{1}{2}} \\
     & \lesssim N^{2-n/2} \langle x - y \rangle^{-n/2}, \qquad \tau \in \left[ 0, \frac{1}{N^2}\right]  \\
      |[\varphi_{> 1} P^+_N e^{-i\tau\Delta} \kappa_{\le \frac{1}{2}}] (x, y) |& \lesssim N^2 \langle N^2 \tau + N|x| - N|y| \rangle^{-n} \kappa_{|x| > 1} \kappa_{|y| \le \frac{1}{2}}, \\
     & \lesssim N^2 \langle N^2\tau + N|x| + N|y| \rangle^{-n} \kappa_{|x| > 1} \kappa_{|y| \le \frac{1}{2}} \\
     & \lesssim N^{2-n/2} \langle x - y \rangle^{-n/2}, \qquad \tau \in \left[\frac{1}{N^2}, \frac{1}{N}\right],  
  \end{align*}
  for any $n > 0$, where $\kappa_{\le \frac{1}{2}}$ is a characteristic function. So Young's inequality tells us 
  \begin{align*}
  \| (\ref{decomp2-1-1}) \|_{L^2_x}& = 2\left\| \int^{\frac{1}{N}}_0  \varphi_{>1} P^+_N e^{-i\tau \Delta} \kappa_{\le \frac{1}{2}} \nabla \phi(t+\tau) \cdot \nabla \varphi_{\le \frac{1}{2}} d\tau  \right\|_{L^2_x} \\
  &\lesssim \frac{1}{N} \left\| \int [\varphi_{>1} P^+_N e^{-i\tau \Delta} \kappa_{\le \frac{1}{2}}](x, y) (\nabla \phi(t+\tau) \cdot \nabla \varphi_{\le \frac{1}{2}})(y) dy \right\|_{L^\infty_\tau L^2_x ([0, \frac{1}{N}] \times \real^2)} \\
  & \lesssim N^{-10} \| \nabla \phi(t+\tau) \cdot \nabla \varphi_{\le \frac{1}{2}} \|_{L^\infty_\tau L^2_x ([0, \frac{1}{N}] \times \real^2)} \\
  &\lesssim N^{-10} \| \varphi_{>\frac{1}{4}} \partial_r \phi\|_{L^\infty_t L^2_x ([0, \infty) \times \real^2)} \lesssim_\phi 1.
  \end{align*}
  The other terms have similar estimate, so on the whole,
  \[ \|(\ref{decomp2-1})\|_{L^2_x} \lesssim N^{-9}. \]
  \cu{Estimate of (\ref{decomp2-2}).}   
  
  We use nonlinearity estimate (\ref{nonest1}) and (\ref{nonest2}) here.
  \begin{align*}
    \| (\ref{decomp2-2})\|_{L^2_x} &\lesssim \| \tilde{P}_N \varphi_{> \frac{1}{2}} F(\phi) \|_{L^1_\tau L^2_x([t, t+\frac{1}{N}] \times \real^2)} \\
    &\lesssim \frac{1}{N}  \| \tilde{P}_N \varphi_{> \frac{1}{2}} F(\phi) \|_{L^\infty_\tau L^2_x([0, \infty) \times \real^2)} \\
    & \lesssim \frac{1}{N^2}  \| \tilde{P}_N |\nabla| (\varphi_{> \frac{1}{2}} F(\phi)) \|_{L^\infty_\tau L^2_x([0, \infty) \times \real^2)} \\
    & \lesssim \frac{1}{N^2} \bigg[ \left\| \varphi_{> \frac{1}{4}} F(\phi) \right\|_{L^\infty_\tau L^2_x([0, \infty) \times \real^2)} + 
      \left\| \varphi_{> \frac{1}{4}}\frac{1}{r} F(\phi) \right\|_{L^\infty_\tau L^2_x([0, \infty) \times \real^2)} 
      \\ &\quad + \left\| \varphi_{> \frac{1}{4}} \partial_r F(\phi) \right\|_{L^\infty_\tau L^2_x([0, \infty) \times \real^2)}\bigg]
     \\
     &\lesssim_\phi \frac{1}{N^2}
  \end{align*}
  \cu{Estimate of (\ref{decomp2-3}).}
  
  This is still a tail term, so we estimate similarly as (\ref{decomp2-1}). Since
  \begin{align*}
   \| (\ref{decomp2-3})\|_{L^2_x} \le& \limsup_{T \rightarrow +\infty} \left\| i\varphi_{>1}\int_{N^{-1}}^T P^+_N e^{-i \tau \Delta} \varphi_{\le \frac{N\tau}{2}} F(\phi(t + \tau)) d\tau \right\|_{L^2_x} \\
   :=& \limsup_{T\rightarrow +\infty} \| (\ref{decomp2-3})_T \|_{L^2_x},
   \end{align*}
   we aim to derive a uniform estimate for $(\ref{decomp2-3})_T$. Using the equation, we have
     \begin{align}
   (\ref{decomp2-3})_T=& i\varphi_{>1}P^+_N \int_{\frac{1}{N}}^T e^{-i\tau \Delta} (\varphi_{\le \frac{N\tau}{2}} (i\partial_\tau + \Delta)\phi(t + \tau)) d\tau \nonumber \\
  =& -\varphi_{>1}P^+_N e^{-iT\Delta} (\varphi_{\le \frac{NT}{2}} \phi(t+T)) \nonumber\\
  & +\varphi_{>1}P^+_N e^{-i\frac{1}{N}\Delta} (\varphi_{\le \frac{1}{2}} \phi(t+\frac{1}{N}))  \nonumber\\
  &- i\varphi_{>1}P^+_N \int_{\frac{1}{N}}^T e^{-i\tau \Delta} \phi(t+\tau) \Delta \varphi_{\le \frac{N\tau}{2}} d\tau \nonumber\\
  &- 2i\varphi_{>1}P^+_N \int_{\frac{1}{N}}^T e^{-i\tau \Delta} \nabla \phi(t+\tau) \cdot \nabla \varphi_{\le \frac{N\tau}{2}} d\tau \nonumber\\
  & - \varphi_{>1} P^+_N \int^T_{\frac{1}{N}} e^{-i\tau\Delta} \left( \frac{r}{N\tau^2} \partial_r \varphi_{\le \frac{1}{2}}\left( \frac{r}{N\tau} \right)\phi(t+\tau)  \right) d \tau \label{decomp2-3-1}
  \end{align}
  Again we have decay estimate of the kernel for $\tau \ge \frac{1}{N}$
        \begin{align*}
      |[\varphi_{> 1} P^+_N e^{-i\tau\Delta} \kappa_{\le \frac{N\tau}{2}}] (x, y) |& \lesssim N^d \langle N^2\tau + N|x| - N|y| \rangle^{-n} \kappa_{|x| > 1} \kappa_{|y| \le \frac{N\tau}{2}}, \\
     & \lesssim N^2 \langle N^2\tau + N|x| + N|y| \rangle^{-n} \kappa_{|x| > 1} \kappa_{|y| \le \frac{N\tau}{2}} \\
    &  \lesssim N^{2-n} (N\tau)^{-n/2} \langle x - y \rangle^{-n/2}.
       \end{align*}
      \begin{align*}
       & \| (\ref{decomp2-3-1})\|_{L^2_x} \\
        =& \left\| \int_{\frac{1}{N}}^T \int_{\real^2} \varphi_{>1} P^+_N e^{-i\tau \Delta} \kappa_{\le \frac{N\tau}{2}} (x, y)\left( \frac{|y|}{N\tau^2} \partial_r \varphi_{\le \frac{1}{2}}\left( \frac{|y|}{N\tau} \right)\phi(t+\tau, y)  \right) dy d\tau   \right\|_{L^2_x} \\
   \lesssim& N^{-10} \int_{\frac{1}{N}}^T (N\tau)^{-6}  \left\|  \left( \frac{|y|}{N\tau^2} \partial_r \varphi_{\le \frac{1}{2}}\left( \frac{|y|}{N\tau} \right)\phi(t+\tau, y)  \right)   \right\|_{L^2_x ([0, \infty) \times \real^2)}d\tau \\
  \lesssim &N^{-10}  \int_{\frac{1}{N}}^T (N\tau)^{-6} \tau^{-1} \left\| \phi\right\|_{L^\infty_{t} L^2_x ([0, \infty) \times \real^2)}d\tau \\
  \lesssim &N^{-9}.
    \end{align*}
    
    To conclude, we have
    \[ \| (\ref{decomp2-3}) \|_{L^2_x} \le \limsup_{T \rightarrow +\infty} \| (\ref{decomp2-3})_T \|_{L^2_x} \lesssim N^{-9}. \]
      
  \cu{Estimate of (\ref{decomp2-4}).}
  
    We first further decompose (\ref{decomp2-4}) by decomposing $ \varphi_{> \frac{N\tau}{2}}  =  \varphi_{> \frac{N\tau}{4}}   \varphi_{> \frac{N\tau}{2}} $ and then introducing a frequency projection
    \begin{align}
      (\ref{decomp2-4}) =  i\varphi_{>1} \int_{N^{-1}}^\infty P^+_N e^{-i \tau \Delta} \varphi_{> \frac{N\tau}{4}} P_{> \frac{N}{8}} \varphi_{> \frac{N\tau}{2}} F(\phi(t + \tau)) d\tau \label{decomp2-4-1}\\
      +  i\varphi_{>1} \int_{N^{-1}}^\infty P^+_N e^{-i \tau \Delta}  \varphi_{> \frac{N\tau}{4}} P_{\le \frac{N}{8}} \varphi_{> \frac{N\tau}{2}} F(\phi(t + \tau)) d\tau \label{decomp2-4-2}
    \end{align}
    (\ref{decomp2-4-1}) can be estimated by weighted Strichartz estimate Lemma \ref{weightedstrichartz}, and estimate of nonlinearity (\ref{nonest3}) and (\ref{nonest4}). 
  \begin{align*}
   \| (\ref{decomp2-4-1}) \|_{L^2_x} 
 \lesssim & \left\| \int_\frac{1}{N}^\infty e^{-i\tau \Delta}\varphi_{> \frac{N\tau}{4}} P_{> \frac{N}{8}} \varphi_{> \frac{N\tau}{2}}F(\phi(t+\tau)) d\tau \right\|_{L^2_x} \\
     \lesssim & \left\| |x|^{-\frac{1}{2}}\varphi_{> \frac{N\tau}{4}} P_{> \frac{N}{8}} \varphi_{> \frac{N\tau}{2}} F(\phi(t+\tau)) \right\|_{L^{\frac{4}{3}}_\tau L^1_x ([\frac{1}{N}, \infty) \times \real^2)}  \\
     \lesssim & \left\|  (N\tau)^{-\frac{1}{2}} \left\| P_{> \frac{N}{8}} \varphi_{> \frac{N\tau}{2}} F(\phi(t+\tau)) \right\|_{L^1_x} \right\|_{L^{\frac{4}{3}}_\tau([\frac{1}{N}, \infty))} \\
     \lesssim & N^{-\frac{3}{2}} \left\|  \tau^{-\frac{1}{2}} \left\| P_{> \frac{N}{8}} |\nabla| \varphi_{> \frac{N\tau}{2}} F(\phi(t+\tau)) \right\|_{L^1_x } \right\|_{L^{\frac{4}{3}}_\tau([\frac{1}{N}, \infty))} \\
    \lesssim&  N^{-\frac{3}{2}} \Bigg\|  \tau^{-\frac{1}{2}} 
    \bigg[ \bigg\| \varphi_{> \frac{N\tau}{4}} (N\tau)^{-1}F(\phi(t+\tau)) \bigg\|_{L^1_x }
    + \bigg\| \varphi_{> \frac{N\tau}{4}}\frac{1}{r} F(\phi(t + \tau)) \bigg\|_{L^1_x } \\    
     & \quad  + \bigg\| \varphi_{> \frac{N\tau}{4}}\partial_r (F(\phi(t + \tau))) \bigg\|_{L^1_x } \bigg]
       \Bigg\|_{L^{\frac{4}{3}}_\tau([\frac{1}{N}, \infty))} \\
     \lesssim  &N^{-\frac{3}{2}} \left\|  \tau^{-\frac{1}{2}} (N\tau)^{-\frac{1}{2}} \right\|_{L^{\frac{4}{3}}_\tau([\frac{1}{N}, \infty))} \lesssim N^{-\frac{7}{4}}.
  \end{align*}
  And (\ref{decomp2-4-2}) easily follows from mismatch estimate Lemma \ref{mismatch2} and (\ref{nonest1}).
  \begin{align*}
     \| (\ref{decomp2-4-2}) \|_{L^2_x} & \lesssim \| \tilde{P}_N \varphi_{> \frac{N\tau}{4}} P_{\le \frac{N}{8}} \varphi_{> \frac{N\tau}{2}} F(\phi(t + \tau)) \|_{L^1_\tau L^2_x ([\frac{1}{N}, \infty) \times \real^2} \\
    & \lesssim N^{-10} \| (N\tau)^{-10} \|_{L^1_\tau([\frac{1}{N}, \infty))} \| \varphi_{> \frac{1}{4}} F(\phi) \|_{L^\infty_t L^2_x([0, \infty) \times \real^2)} \\
    & \lesssim N^{-11}.
  \end{align*}
     Thus we've proved 
     \[ \| (\ref{decomp2-4}) \|_{L^2_x}\lesssim N^{-\frac{7}{4}}. \]  
  Collecting estimates of (\ref{decomp2-1})-(\ref{decomp2-4}), we've shown $\|(\ref{decomp1-2}) \|_{L^2_x} \lesssim N^{-1-\frac{3}{4}} $. Similar estimate holds for (\ref{decomp1-3}), and thus we've completed the proof of Proposition \ref{freqdecayest}.
  
\end{proof}

Finally let's complete the proof of nonlinear estimates, based on charge conservation and the weak localization of kinetic energy  in  Proposition \ref{weaklocal}.

\begin{proof}[Proof of Lemma \ref{nonlinearest}]
    Note that $F (\phi) = F(u)e^{im\theta}$. For simplicity, we replace $F(\phi)$ by $F(u)$ as the target estimates will remain the same. Define 
    \begin{align*}
     A_0^{(1)}u := - u\int_r^\infty \frac{A_\theta}{s^2} |u|^2 sds, \quad  A_0^{(2)}u := -u \int_r^\infty \frac{m}{s^2} |u|^2 sds
    \end{align*}
    so that $A_0 = A_0^{(1)} + A_0^{(2)}$. And then define 
    \begin{align*}
      N_1 (u) = -|u|^2 u,\quad N_2 (u) = \frac{2m}{r^2}A_\theta u, \quad N_3(u) = \frac{A_\theta ^2}{r^2} u,\\
      N_4(u) := A_0^{(1)}u, \quad   N_5(u) := A_0^{(2)}u. \quad\quad\,\,\,
    \end{align*}
    Thus $F(u) = \sum_{i=1}^5 N_i(u)$. We will deal with each $N_i(u)$ separately. 
    
    Under the assumption of Proposition \ref{freqdecayest}, we know $\phi(t)$ satisfies Proposition \ref{weaklocal} hence 
    \beq \| \phi_{> \frac{1}{16}} \nabla \phi \|_{L^\infty_t L^2_x([0, \infty) \times \real^2)} \lesssim_\phi 1. \eeq
    And hence by Strauss estimate (\ref{strauss1}), for $r \ge \frac{1}{16}$, 
  \begin{align}
    \| u(t, r) \|_{L^\infty_t([0, \infty))} &\le \| \partial_r u(t) \|_{L^\infty_t L^2_x([0, \infty) \times \{ |x| \ge \frac{1}{16}\})}^{\frac{1}{2}} \| u \|_{L^\infty_t L^2_x([0, \infty) \times \{ |x| \ge \frac{1}{16}\})}^{\frac{1}{2}} r^{-\frac{1}{2}} \notag\\&\lesssim_u r^{-\frac{1}{2}} \label{linftydecay}
\end{align}
     In particular,
     \beq \|\varphi_{> \frac{1}{8}} u\|_{L^\infty_t L^\infty_x([0,\infty) \times \real^2)} \lesssim_u 1. \label{linftybdd}\eeq
     In the following estimate, our tools will be  (\ref{linftydecay}) (\ref{linftybdd})  and conservation of charge 
     \beq \| u \|_{L^\infty_t L^2_x([0, \infty) \times \real^2)} \lesssim_u 1.\label{l2uniformdecay} \eeq
     We begin with estimate on $A_\theta$, $A_0$.
     \begin{align}
     \|A_\theta \|_{L^\infty_t L^\infty_x} &\lesssim \| u\|^2_{L^\infty_t L^2_x} \lesssim_u 1, \\
     \|\varphi_{> \frac{1}{16}}A_0^{(2)} r^{2} \|_{L^\infty_t L^\infty_x} &\lesssim  \| \varphi_{> \frac{1}{8}} u\|^2_{L^\infty_t L^2_x} \lesssim_u 1,\label{a0decay1}\\
     \|\varphi_{> \frac{1}{8}}A_0^{(1)} r^2 \|_{L^\infty_t L^\infty_x} &\lesssim \|A_\theta \|_{L^\infty_t L^\infty_x} \|\varphi_{> \frac{1}{8}}A_0^{(2)} r^2 \|_{L^\infty_t L^\infty_x} \lesssim_u 1. \label{a0decay2}
     \end{align}
     
     Now (\ref{nonest1}) and (\ref{nonest2}) easily follows these bounds. We show estimates of $N_1$ and $N_2$ as examples. The following estimates are valid uniformly in time.
     \begin{align*}
       \| \varphi_{> \frac{1}{4}} N_1(u) \|_{L^2_x}& \le \| \varphi_{> \frac{1}{8}} u \|_{L^2_x} \| \varphi_{> \frac{1}{8}} u \|^2_{L^\infty_x} \lesssim_u 1,  \\
       \| \varphi_{> \frac{1}{4}} \partial_r (N_1(u)) \|_{L^2_x}& \lesssim \| \varphi_{> \frac{1}{8}} \partial_r u \|_{L^2_x} \| \varphi_{> \frac{1}{8}} u \|^2_{L^\infty_x} \lesssim_u 1, \\
       \| \varphi_{> \frac{1}{4}} N_2(u) \|_{L^2_x} &\le \left\| \varphi_{> \frac{1}{8}} u\frac{1}{r^2} \right\|_{L^2_x} \| A_\theta \|_{L^\infty_x}  \lesssim \| \varphi_{> \frac{1}{8}} u \|_{L^2_x} \| u \|_{L^2_x} \lesssim_u 1, \\
        \| \varphi_{> \frac{1}{4}}\partial_r (N_2(u)) \|_{L^2_x} &\lesssim \left\| \varphi_{>\frac{1}{4}}  \frac{|u|^2 u}{r^2} \right\|_{L^2_x} + \left\| \varphi_{>\frac{1}{4}}  \frac{A_\theta \partial_r u}{r^2} \right\|_{L^2_x} + \left\| \varphi_{>\frac{1}{4}} \frac{A_\theta u}{r^3} \right\|_{L^2_x}\\
        &\le  \| \varphi_{> \frac{1}{4}} N_1(u) \|_{L^2_x} + \| A_\theta\|_{L^\infty_x} \| \varphi_{> \frac{1}{8}} \partial_r u \|_{L^2_x} \lesssim_u 1.\\
     \end{align*}
     Next, for (\ref{nonest3}) and (\ref{nonest4}), we need to be more careful so as to gain enough spatial decay. So here we use (\ref{linftydecay}) instead of (\ref{linftybdd}). Firstly, for $N_1(u)$ 
     \begin{align*}
       \| \varphi_{> T} N_1(u) \|_{L^1_x}& \le \| \varphi_{> \frac{T}{2}} u \|^2_{L^2_x} \| \varphi_{> \frac{T}{2}} r^{\frac{1}{2}}u \|_{L^\infty_x} \| \varphi_{> \frac{T}{2}} r^{-\frac{1}{2}} \|_{L^\infty_x} \lesssim_u T^{-\frac{1}{2}},\\ 
       \| \varphi_{> T} \partial_r (N_1(u)) \|_{L^1_x}& \le \| \varphi_{> \frac{T}{2}} \partial_r u \|_{L^2_x} \| \varphi_{> \frac{T}{2}} u \|_{L^2_x} \| \varphi_{> \frac{T}{2}} r^{\frac{1}{2}}u \|_{L^\infty_x} \| \varphi_{> \frac{T}{2}} r^{-\frac{1}{2}} \|_{L^\infty_x}  \\&\lesssim_u T^{-\frac{1}{2}}.
     \end{align*}
     For other nonlocal nonlinearities, we make use of their $r^{-2}$ spatial decay (see (\ref{a0decay1}), (\ref{a0decay2}) for decay of $A_0$). We merely show estimates for $N_2(u)$ and the others comes in the similar way.
     \begin{align*}
     \| \varphi_{> T} N_2(u) \|_{L^1_x}& \lesssim \| A_\theta  \|^2_{L^\infty_x} \| \varphi_{> \frac{T}{2}} u \|_{L^2_x} \| \varphi_{> \frac{T}{2}} r^{-2} \|_{L^2_x} \lesssim_u T^{-1},\\
      \| \varphi_{> T}\partial_r (N_2(u)) \|_{L^1_x} &\lesssim \left\| \varphi_{>T}  \left(\frac{N_1(u)}{r^2} + \frac{A_\theta \partial_r u}{r^2} + \frac{N_2(u)}{r}\right) \right\|_{L^1_x}\\
      & \lesssim_u T^{-\frac{5}{2}} + T^{-1} + T^{-2} \lesssim T^{-1}.
     \end{align*}
     
\end{proof}

This completes the proof of Proposition \ref{freqdecayest}, and thus Theorem \ref{rigidinfselfdual} with arguments in \S \ref{4.2}.\newline

\textbf{Acknowledgement.} We are grateful to Kihyun Kim and Soonsik Kwon for helpful discussions and remarks. The authors are supported by the NSF of China (No. 12071010, 11631002).

\appendix
\section{Elliptic theory for CSS}\label{AppA}
\subsection{Self-dual case}

\begin{proof}[Proof of Proposition \ref{varcharselfdual}]
  The positivity of energy follows easily with 
  \beq E[\phi_0] = \frac{1}{2}\int | \bm{D}_+ \phi_0 |^2 \ge 0.
  \eeq
  For the rigidity of null energy solution, note that $E[\phi_0] = 0$ implies $D_+ \phi_0 =0$, i.e.
  \beq \partial_r \phi_0 - \frac{m+A_\theta[\phi_0]}{r} \phi_0 = 0, \qquad \forall r > 0.
  \eeq
  Since $A_\theta[u]$ is real, this implies its radial part $u$ satisfies
  \beq \partial_r |u|^2 - \frac{2m + 2A_\theta[u]}{r} |u|^2 = 0. \label{selfdualeq} \eeq
  We claim that $u$ is not non-zero function implies $|u| > 0$ for all $r > 0$. Otherwise, there exists $r_0 \in (0, \infty)$ such that $u(r_0) = 0$, then Gronwall's inequality shows
  \[ |u|^2(r) \le \text{exp}\left(\int^r_{r_0} \frac{2m + 2 A_\theta[u](r')}{r'} dr' \right)|u|^2(r_0) = 0,\qquad \forall\, r >0, \]
   hence $|u| \equiv 0$, which is a contradiction. 

  Then, by representing value of $u$ in polar coordinates, we see that the phase part is always constant, i.e. $u(r) = \rho(r) e^{i \gamma_0}$ for $\rho$ some real-valued function, and without loss of generality, we can assume $\rho > 0$.
  
  Next by change of coordinates $v := r^{-2m}\rho^2$, $w:=\text{log } v$. From (\ref{selfdualeq}), we see $w$ satisfies the second-order ODE on $[0, +\infty)$ with one boundary condition.
  \begin{align*}
    w'' + \frac{1}{r}w' + r^{2m}e^w &=0\\
    w'(0) &= 0
  \end{align*}
 Each initial value $w(0) \in \real$ uniquely determines a solution, the corresponding $\rho$ of which is the scaling of the soliton $Q$. Combined with local wellposedness theory, we complete the proof.
\end{proof}

\subsection{Non-self-dual case}

Recall that when $g > 1$, the energy here is 
\beq E[\phi_0] = \frac{1}{2} \| \bm{D}_x \phi_0 \|_{L^2}^2 - \frac{g}{4} \| \phi_0\|_{L^4}^4 =  \frac{1}{2} \| \bm{D}_+ \phi_0 \|_{L^2}^2 - \frac{g-1}{4} \| \phi_0\|_{L^4}^4. \eeq
The threshold charge is defined as minimization of $L^2$ norm of m-equivariant function with non-positive energy.
\begin{lem}[{\cite[Lemma 7.2, Lemma 7.5]{liu2016global}}]\label{lema1}
Define 
\beq c_{m,g} = \inf \{ \| \phi \|_{L^2} : \phi \in H^1_m \backslash \{ 0 \}, E[\phi] \le 0\}, \label{nsd2} \eeq
Then $c_{m, g} > 0$, and
\beq c_{m, g}= \inf \{ \| \phi \|_{L^2} : \phi \in H^1_m \backslash \{ 0 \}, E[\phi] = 0\}. \label{nsd1} \eeq
\end{lem}

The minimizer of (\ref{nsd1}) is a standing wave solution to (\ref{CSS}). 

\begin{lem}[{\cite[Lemma 7.7]{liu2016global}}]\label{lema2}
  Let $\phi \in H^1_m \backslash \{ 0 \}$  with $E[\phi] = 0, \|\phi\|_{L^2} = c_{m, g} $. Then there exists $\alpha \in \real$ such that $\psi(t,x) := e^{i\alpha t} \phi(x)$ is a solution of (\ref{CSS}).
\end{lem}
\begin{rmk}
  This result comes from minimization problem of charge with restraint $E[\phi] = 0$. If $E'[\phi] = 0$, we get the case $\alpha = 0$, and otherwise the minimizer satisfies Euler-Lagrangian equation with $\alpha \neq 0$. We can further exclude the case $\alpha < 0$ by  \cite[Proposition 4.2]{byeon2012gaugeNLS} and   \cite[Proposition 3.3]{byeon2016standing} .
\end{rmk}

Next we exclude the charge minimizer with negative energy. 
\begin{lem}\label{lema3}
  If $\phi \in H^1_m$ with $ E[\phi] < 0 $, then $\| \phi\|_{L^2} > c_{m, g}$.
\end{lem}
\begin{proof}
  From (\ref{nsd2}), we know  $\| \phi\|_{L^2} \ge c_{m, g} $. So by contradiction, suppose there exists $\phi \in H^1_m$ such that $E[\phi] < 0$ and $\|\phi \|_{L^2} = c_{m, g} $. Note that for $\beta \in \real$, 
  \[ E[\beta \phi] = \frac{\beta^2}{2} \int_{\real^2} \left[ |\partial_r \phi|^2 + \frac{1}{r^2} \left( m^2 - \beta^2 \frac{1}{2} \int^r_0 |\phi|^2 sds \right)^2 |\phi|^2 - \beta^2 \frac{g}{2} |\phi|^4 \right] dx \] 
  continuously depends on $\beta$, so we can choose $\beta < 1$ very close to $1$, such that $E[\beta \phi] < 0$, then \[ \| \beta \phi\|_{L^2} = \beta c_{m, g} < c_{m, g}\]
  which contradicts to (\ref{nsd2}).
\end{proof}

Combining these two lemmas, we get the variational characterization of null-energy critical-charge solution in non-self-dual case.

\begin{proof}[Proof of Proposition \ref{varcharnonselfdual}]
The positivity of energy comes directly from Lemma \ref{lema3}. For the rigidity, only to note that if $E[\phi_0] = 0$ and $\phi_0$ is non-zero, by Lemma \ref{lema1} we also have $\| \phi_0\|_{L^2} \ge c_{m, g} $, while the condition says $\| \phi_0 \|_{L^2} \le c_{m,g}$. So the equivalence holds, and Lemma \ref{lema2} ensures $\phi_0$ to generate a stationary wave or static solution. 

Finally we prove the exponential decay of any stationary wave satisfying (\ref{station}) with $\alpha > 0$. Take 
\[ c_u (x) := \alpha + \frac{m^2}{r^2} + \frac{2m}{r^2}A_\theta[u] + A_0[u] + \frac{1}{r^2} A_\theta [u]^2 - g|u|^2. \]
From Stauss estimate $|u(r)| \lesssim r^{-1}$ and $A_0[u](r) \rightarrow 0$ as $r \rightarrow \infty$, we have $\mathcal{L} u := (\Delta - c_u(x)) u = 0$ and $c_u(x) \in (\frac{\alpha}{2}, \frac{3\alpha}{2})$ on $B_R^c$ with $R$ large enough. So we can construct $v := e^{-a|x|}$ with $a$ small enough such that $\mathcal{L}v \le (\Delta - \frac{\alpha}{2}) v \le 0$. Now we take $C$ large enough such that $u-v < 0$ on $\partial B_R$, and the comparison theorem implies that $u \le Cv = Ce^{-a|x|}$.
\end{proof}

\section{Covariant $H^1_m$ norm and equivariant Sobolev space}\label{AppB}
We denote $B_r := \{ x \in \real^2 : |x| < r \}$ and $B_r^c = \real^2 \backslash B_r$ throughout this section.
  \begin{proof}[Proof of Lemma \ref{equilem1}]
  Note that $A_\theta[f](r) = -\frac{1}{4\pi} \| f\|^2_{L^2(B_r)}$ is decreasing. If $\|f\|_{L^2}^2 \le 2\pi m$, then $A_\theta[f] \ge -\frac{m}{2}$ and obviously we have (\ref{lem311}). Otherwise, take 
  \[ R := \sup_{r \ge 0} \left\{ A_\theta[f](r) \ge  -\frac{m}{2}\right\} \in (0, \infty). \]
  Then we have 
  \[ \| f \|^2_{L^2(B_R)} = 2\pi m. \]
  And
  \begin{align}
    \left\| \frac{1}{r}(m + A_\theta[f]) f \right\|_{L^2(B_R)}^2 &\ge \frac{m^2}{4} \left\| \frac{1}{r} f \right\|_{L^2(B_R)}^2 \label{b1} \\
    \left\| \frac{1}{r}(m + A_\theta[f]) f \right\|_{L^2(B_R)}^2 &\ge \frac{1}{R^2} \| f \|^2_{L^2(B_R)} \frac{m^2}{4} = \frac{\pi m^3}{2R^2} \label{b2}\\
    \left\| \frac{1}{r} f \right\|_{L^2(B_R^c)}^2 &\le \frac{1}{R^2} \| f\|_{L^2}^2 \label{b3}
  \end{align}
  (\ref{b2}) and (\ref{b3}) implies that 
  \beq \left\| \frac{1}{r} f \right\|_{L^2(B_R^c)}^2 \le \frac{2\|f\|_{L^2}^2}{\pi m^3} \left\| \frac{1}{r}(m + A_\theta[f]) f \right\|_{L^2(B_R)}^2 \label{b4}\eeq
  The estimate (\ref{lem311}) for $\|f\|_{L^2}^2 \ge 2\pi m$ is established by combining (\ref{b1}) and (\ref{b4}).
  
  For (\ref{lem312}), we use Strauss' estimate (\ref{strauss1}) to estimate $A_\theta$: for $f \in H^1_m$, 
  \beq  |A_\theta[f] (r)| \le \frac{1}{2} \int_0^r \| f r^{\frac{1}{2}}\|_{L^\infty}^2 ds \lesssim \| \partial_r f \|_{L^2} \| f \|_{L^2} r.  \label{b5} \eeq
  Thus 
  \[ \left\|  \frac{A_{\theta}[f]}{r} f \right\|_{L^2} \le \left\|  \frac{A_{\theta}[f]}{r}  \right\|_{L^\infty}  \|f \|_{L^2} \lesssim \|f\|_{L^2}^2 \|\partial_r f\|_{L^2}. \]
  
  Combine (\ref{lem311}), (\ref{lem312}), we get (\ref{lem313}). For $m > 0$, 
  \begin{align*}
    \| f \|^2_{\dot{H}^1_m}& = \| \partial_r f\|_{L^2}^2 + \left\|\frac{m}{r}f\right\|_{L^2}^2 \\
    & \lesssim_{m, \|f\|_{L^2}} \| \partial_r f\|_{L^2}^2 +  \left\| \frac{m+A_{\theta}[f]}{r} f \right\|_{L^2}^2 = \| \bm{D}_x f\|^2_{L^2}, \\
     \| \bm{D}_x f\|^2_{L^2}& = \| \partial_r f\|_{L^2}^2 +  \left\| \frac{m+A_{\theta}[f]}{r} f \right\|_{L^2}^2 \\
     &\le  \| \partial_r f\|_{L^2}^2 + (1 + \frac{1}{4\pi m} \|f\|_{L^2}^2)\left\| \frac{m}{r} f\right\|_{L^2}^2 \lesssim  \| f \|^2_{\dot{H}^1_m}.
  \end{align*}
  For $m=0$, the "less than" direction is immediate, while the other comes from (\ref{lem312}). 
\end{proof}

\begin{proof}[Proof of Lemma \ref{equilem2}]
   Firstly, since 
   \[ \partial_r = \frac{x_1}{|x|} \partial_1 + \frac{x_2}{|x|} \partial_2,\]
   and $ \frac{x_1}{|x|},  \frac{x_2}{|x|}$ are $L^\infty$ function, we immediately have (\ref{lem321}) from the weak convergence of $\partial_1 f_n$ and $\partial_2 f_n$.
   
   Next we show (\ref{lem323}) for $m \ge 1$. 
   
   Let $\| f_n \|_{H^1}, \|f\|_{H^1} \le M $. Note that $f_n \rightharpoonup f$ in $L^2$, and that $\frac{1}{r}$ is bounded away from the origin, we have $\forall R > 0$
   \beq \frac{1}{r} f_n \rightharpoonup \frac{1}{r} f \qquad \text{in}\,\, L^2(B_R^c), \label{b6}\eeq
   Now we can show that for any $g \in L^2$, 
   \beq \left(\frac{1}{r} f_n , g\right)_{L^2} \rightarrow \left(\frac{1}{r} f, g\right)_{L^2} ,\qquad \text{as}\,\,n \rightarrow \infty.\label{b7}\eeq
   Notice that $\frac{1}{r} f_n, \frac{1}{r}f$ are uniformly bounded in $L^2$ by M (for $m \ge 1$). We then have, $\forall R > 0$, from (\ref{b6})
   \begin{align*}
     \left| \left( \frac{1}{r}(f_n - f), g \right)_{L^2} \right| & \le  \left| \left( \frac{1}{r}(f_n - f), g \right)_{L^2(B_R^c)} \right| +  \left| \left( \frac{1}{r}(f_n - f), g \right)_{L^2(B_R)} \right| \\
     & \le o_n(1) + 2M \| g \|_{L^2(B_R)}.
   \end{align*} 
   From the arbitrariness of $R > 0$, (\ref{b7}) is confirmed. That's the weak convergence in (\ref{lem323}). 
   
   Finally, we prove (\ref{lem322}) for $m \ge 0$. (\ref{lem324}) directly follows (\ref{lem321})-(\ref{lem323}).
   
   Note that $\frac{A_\theta[f_n]}{r}, \frac{A_\theta[f]}{r}$ are uniformly $L^\infty$ bounded away from the origin,  using same strategy as above, we only need to prove the following two things
   \begin{enumerate}[(1)]
     \item $ \frac{A_\theta[f_n]}{r} f_n $, $\frac{A_\theta[f]}{r} f$ are uniformly bounded in $L^2$.
     \item For all $R > 0$, 
     \[ \frac{A_\theta[f_n]}{r} f_n \rightharpoonup \frac{A_\theta[f_n]}{r} f \qquad \text{in}\,\,B_R^c. \]
   \end{enumerate}
   (1) follows (\ref{lem312}) in Lemma \ref{equilem1}, noticing that $f_n$, $f$ are uniformly bounded in $H^1$.
   
   For (2), we again take a $g \in L^2$ and try to derive the convergence of the inner product. From the $m$-equivariance assumption, without loss of generosity, we may take $g \in L^2_{\text{rad}}$. Moreover, we can take $g \in C^\infty_{c, \text{rad}}$ as a test function by the density argument.
   \begin{align*}
     \left( \frac{1}{r}A_\theta[f_n] f_n, g \right)_{L^2(B^c_R)}  &=  \int_R^\infty \frac{1}{r} f_n(r) \bar{g}(r) \int_0^r |f_n|^2(s) sds \, rdr   \\ 
     &= \int_0^\infty |f_n|^2(s) s \left( \int_{\max\{R, s\}}^\infty f_n(r) \bar{g}(r) dr\right)ds \\
     &=: \int_0^\infty |f_n|^2(s) s\, G_{n, R} (s) ds
   \end{align*}
   Then from the weak convergence (\ref{b6}) for general $m \ge 0$, we have a pointwise convergence
   \[ G_{n, R}(s) = \left( \frac{1}{r} f_n, g \right)_{L^2(B^c_{\max\{R, s\}})}  \rightarrow  \left( \frac{1}{r} f, g \right)_{L^2(B^c_{\max\{R, s\}})} =: G_R(s), \forall s \ge 0.  \]
   So $G_{n, R}$ and $G_R(s)$ are uniformly bounded in $L^\infty$. And compact support of $g$ implies that $G_{n, R}, G_R$ also compactly supported. So They are uniformly bounded in $L^p$ for any $p \in [1, \infty]$. By dominated convergence,
   \[ G_{n, R} \rightarrow G_R,\qquad \text{in} \,\, L^2. \]
   Also from compact embedding $H^1_{\text{rad}}(\real^2) \hookrightarrow L^4(\real^2)$, we have
   \[ f_n \rightarrow f,\qquad \text{in}\,\, L^4. \]
   Now 
   \begin{align*}
      & \left| \left( \frac{1}{r}A_\theta[f_n] f_n, g \right)_{L^2(B^c_R)} - \left( \frac{1}{r}A_\theta[f] f, g \right)_{L^2(B^c_R)} \right|\\
     \le & \left| (G_{n, R}, |f_n|^2)_{L^2} - (G_{R}, |f|^2)_{L^2} \right|\\
     \le & \left| (G_{n, R} - G_R, |f_n|^2)_{L^2} \right| +  \left| (G_R f_n, f_n - f)_{L^2} \right| +  \left| (G_R \bar{f}, \overline{f_n - f})_{L^2} \right| \\
     \le & \| G_{n, R} - G_R \|_{L^2} \| f_n \|_{L^4}^2 + \| G_R \|_{L^2} (\|f_n \|_{L^4} + \|f\|_{L^4}) \|f_n - f\|_{L^4}\\   & \rightarrow 0,\qquad \text{as} \,\, n \rightarrow \infty.
   \end{align*}
   That finish the proof of (2) and conclude this lemma.
\end{proof}

\begin{proof}[Proof of Lemma \ref{equilem3}]
  Consider $\bm{D}_+ v_n - \bm{D}_+ v \rightarrow 0$ in $L^2$. Note that
   \begin{align*}
    \| \bm{D}_+ v_n &- \bm{D}_+ v \|_{L^2}^2 = \left\| \partial_r (v_n - v) - \frac{m}{r}(v_n - v) - \left( \frac{A_\theta [v_n] }{r} v_n -\frac{A_\theta [v] }{r} v \right) \right\|_{L^2}^2  \\
    \gtrsim & \left\| \partial_r (v_n - v) - \frac{m + A_\theta [v]}{r}(v_n - v) \right\|_{L^2}^2 -  \left\|  \frac{A_\theta [v_n] - A_\theta [v] }{r} v_n \right\|_{L^2}^2 \\
    :=& \uppercase\expandafter{\romannumeral1}_n - \uppercase\expandafter{\romannumeral2}_n
     \end{align*}
     
     And using integration by parts,
     \begin{align*}
    \uppercase\expandafter{\romannumeral1}_n   =&  \int_0^\infty \Bigg[| \partial_r (v_n - v)|^2 + \left| \frac{m + A_\theta [v]}{r}(v_n - v) \right|^2  
  - 2 \text{Re}\left(\overline{\partial_r (v_n - v)} (v_n - v) \frac{m + A_\theta [v]}{r} \right) \Bigg]rdr  \\
    =& \left\| \partial_r (v_n - v) \right\|_{L^2}^2 +  \left\| \frac{m + A_\theta [v]}{r}(v_n - v) \right\|_{L^2}^2 - \frac{1}{2} \int_0^\infty |v_n -v|^2 |v|^2 rdr  \\
    :=& \uppercase\expandafter{\romannumeral1}^1_n + \uppercase\expandafter{\romannumeral1}^2_n - \frac{1}{2}\uppercase\expandafter{\romannumeral1}_n^3
   \end{align*}
   From convergence in $L^2$ and $L^4$, and the uniform bound in $H^1$ for $v_n$ and $v$, we know 
   \begin{align*}
      \uppercase\expandafter{\romannumeral2}_n &\le \| A_\theta[v_n] - A_\theta[v] \|^2_{L^\infty} \left\|\frac{v_n}{r} \right\|^2_{L^2}\\
      &\lesssim \|v_n - v \|^2_{L^2} \| v_n + v \|^2_{L^2}  \left\|\frac{v_n}{r} \right\|^2_{L^2} \rightarrow 0 \\
       \uppercase\expandafter{\romannumeral1}^3_n &\lesssim \|v_n - v\|_{L^4}^2 \|v \|_{L^4}^2 \rightarrow 0
    \end{align*}
   Thus 
   \beq  \uppercase\expandafter{\romannumeral1}^1_n +  \uppercase\expandafter{\romannumeral1}_n^2 \rightarrow 0, \eeq
   which implies that 
   \beq   \partial_r v_n \rightarrow \partial_r v \qquad \text{in}\,\,L^2.\label{b9}\eeq
   And for $m \ge 1$, take $r_0 > 0$ to be \[ r_0 := \sup_{r \ge 0}\left\{ A_\theta[v](r) \ge - \frac{m}{2}\right\}>0, \]
   we have 
   \[ \frac{1}{r}v_n \rightarrow \frac{1}{r}{v} \qquad \text{in}\,\,L^2(\{|x| \le r_0 \}). \]
   Combined with $v_n \rightarrow v$ in $L^2$, we know for $m \ge 1$,
   \beq \frac{1}{r}v_n \rightarrow \frac{1}{r}{v} \qquad \text{in}\,\,L^2.\label{b10}\eeq
   (\ref{b9}) and (\ref{b10}) conclude the $\dot{H}^1$ convergence of $v_n$ to $v$.
\end{proof}

\bibliographystyle{siam}
\bibliography{Bib-1}

\end{document}